\documentclass[11pt]{article}
\usepackage{color}
\usepackage[margin=1in]{geometry}
\usepackage{amssymb}
\usepackage{amsmath}
\usepackage{amsthm}
\usepackage{mathtools}
\usepackage{graphicx}
\usepackage[normalem]{ulem}
\usepackage{algorithm}
\usepackage{algorithmicx}
\usepackage{algpseudocode}
\usepackage[colorlinks=true]{hyperref}

\newtheorem{theorem}{Theorem}[section]
\newtheorem{lemma}{Lemma}[section]

\newtheorem{corollary}{Corollary}[section]
\newtheorem{remark}{Remark}[section]
\numberwithin{equation}{section}

\allowdisplaybreaks 
\begin{document}
\title{A Tensor Neural Network Method for High-Order Homogenization of Locally Periodic Elliptic Problems\footnote{This work was supported  by the National Key Laboratory of 
Computational Physics of China (6142A05230501), the National Center for 
Mathematics and Interdisciplinary Science, CAS.}}
\author{Huaijia Zhang\footnote{Beijing Key Laboratory on MCAACI, Beijing Institute of Technology, Beijing 100081, China,  and School of Mathematics and Statistics, Beijing Institute of Technology, Beijing 100081, China (3120225733@bit.edu.cn).}, \ \
Haochen Liu\footnote{Corresponding author. LSEC, NCMIS, Institute
of Computational Mathematics, Academy of Mathematics and Systems
Science, Chinese Academy of Sciences, Beijing 100190,
China,  and School of Mathematical Sciences, University
of Chinese Academy of Sciences, Beijing 100049, China (liuhaochen@lsec.cc.ac.cn).}, \ \
Xia Ji\footnote{Beijing Key Laboratory on MCAACI, Beijing Institute of Technology, Beijing 100081, China,  and School of Mathematics and Statistics, Beijing Institute of Technology, Beijing 100081, China (jixia@bit.edu.cn).}
\ \  and \ \
Hehu Xie\footnote{LSEC, NCMIS, Institute
of Computational Mathematics, Academy of Mathematics and Systems
Science, Chinese Academy of Sciences, Beijing 100190,
China,  and School of Mathematical Sciences, University
of Chinese Academy of Sciences, Beijing 100049, China (hhxie@lsec.cc.ac.cn).}}

\date{}
\maketitle

\begin{abstract}
We develop a high-order tensor-neural-network method for locally periodic elliptic multiscale problems of the form
\[
-\nabla\cdot\left(A\left(x,\frac{x}{\varepsilon}\right)\nabla u_\varepsilon\right)=f,
\]
where the coefficient depends both on the slow variable $x$ and on the fast periodic variable $y=x/\varepsilon$. Compared with the classical periodic case $A=A(y)$, the locally periodic setting leads to a substantially more involved hierarchy of high-order cell problems and macroscopic corrector equations, since the correctors depend parametrically on the slow variable. We first derive a computable high-order two-scale expansion for $A(x,y)$ and prove an $H^1$ convergence estimate for the partial expansion in boundary-layer-free settings, such as periodic domains or ideal boundary-matching configurations. The proof is based on the recursive compatibility structure of the high-order corrector system and a zero-mean oscillation estimate in $H^{-1}$. 

On the numerical side, we propose a tensor neural network (TNN) framework for solving the high-dimensional corrector hierarchy. The tensor product structure allows the high-dimensional integrals arising in cell problems, homogenized coefficients, macroscopic source terms, and loss functions to be evaluated by deterministic one-dimensional quadrature rules, avoiding Monte Carlo integration errors. The numerical realization assumes that the entries of the coefficient and the assembled data are available in a finite or controlled tensor-product representation; this computational assumption is separate from the general matrix-valued coefficient class used in the analysis. This feature is particularly important for high-order homogenization, where the accumulated numerical error in successive corrector equations may otherwise destroy the expected convergence rate. Numerical experiments for scalar locally periodic coefficients demonstrate that the proposed method accurately computes the high-order correctors. The $H^1$ semi-norm results are consistent with the proved boundary-layer-free estimate, while the point-normalized $L^2$ errors exhibit the nominal high-order behavior predicted by the formal expansion.



\end{abstract}

\section{Introduction}
\label{sec:introduction}

Multiscale elliptic equations with rapidly oscillating coefficients arise in many applications, including composite materials, porous media, diffusion in heterogeneous environments, and multiscale continuum mechanics. A prototypical model is
\begin{equation}
-\nabla\cdot\left(A^\varepsilon(x)\nabla u_\varepsilon(x)\right)=f(x),
\qquad
A^\varepsilon(x)=A\left(x,\frac{x}{\varepsilon}\right),
\label{eq:intro_model}
\end{equation}
where $0<\varepsilon\ll1$ denotes the microscopic length scale. When the coefficient is purely periodic, namely $A^\varepsilon(x)=A(x/\varepsilon)$, the mathematical theory of periodic homogenization is classical; see, for example, \cite{BensoussanLionsPapanicolaou, SanchezPalencia, BakhvalovPanasenko, JikovKozlovOleinik, CioranescuDonato, TartarBook, ShenBook}. In this setting, the oscillatory equation is approximated by a homogenized equation with an effective coefficient, and the leading-order corrector is obtained from a family of periodic cell problems. Quantitative convergence rates and corrector estimates have been extensively studied by variational methods, compactness methods, and operator-theoretic approaches; see, among many others, \cite{AvellanedaLin, KenigLinShen, ShenBoundary, SuslinaDirichlet, NiuShenXu}.

The present work is concerned with the more general locally periodic case
\[
A^\varepsilon(x)=A\left(x,\frac{x}{\varepsilon}\right),
\]
where the coefficient is periodic only with respect to the fast variable and depends nontrivially on the macroscopic variable. Such coefficients are natural in applications where the microscopic structure varies slowly in space. From the theoretical viewpoint, this setting is more delicate than the purely periodic case because the cell functions, homogenized coefficients, and high-order correctors all become parameter-dependent functions of the slow variable. Locally periodic elliptic operators have been studied in the framework of qualitative homogenization and quantitative operator estimates; see, for example, \cite{PastukhovaTikhomirov, SenikJMAA, SenikSIAM}. These works provide fundamental convergence and corrector estimates for locally periodic operators, including boundary value problems. However, the construction of a computable high-order asymptotic hierarchy for $A(x,y)$ and its numerical realization in high dimensions remain challenging.

High-order asymptotic expansions are useful when one seeks accuracy beyond the leading homogenized solution. In the periodic case, the formal expansion generates a cascade of cell problems and macroscopic correction equations. Boundary layers play an essential role in bounded domains with Dirichlet boundary conditions. In particular, the standard two-scale expansion generally does not satisfy the boundary condition and must be supplemented by boundary layer correctors; see \cite{BensoussanLionsPapanicolaou, AllaireAmar, MoskowVogelius}. Boundary layer tails may influence higher-order interior error estimates and may lead to effective boundary conditions. In this paper, we deliberately isolate a different issue: the high-order algebraic structure and numerical computation of the locally periodic corrector hierarchy. Therefore, our convergence theorem is formulated in boundary-layer-free settings, such as periodic domains or ideal boundary-matching configurations. The treatment of boundary layer tails for Dirichlet problems is left to future work.

Numerically, solving multiscale problems by direct discretization is prohibitively expensive when $\varepsilon$ is small, since the mesh size must resolve the microscopic oscillations. Classical multiscale methods, including multiscale finite element methods and heterogeneous multiscale methods, reduce this cost by incorporating local cell information into macroscopic discretizations and are highly effective in low and moderate spatial dimensions; see \cite{HouWu, EWeinanEngquist, Abdulle}. In the locally periodic high-order setting considered here, however, the auxiliary correctors depend jointly on the slow and fast variables and must be computed recursively. A conventional grid-based treatment therefore requires either repeated cell solves over many macroscopic samples or a direct discretization on $\Omega\times Y$, both of which become increasingly expensive as the dimension and expansion order grow.

In recent years, machine learning methods have been increasingly used for multiscale PDEs and homogenization. Neural-network-based solvers have been designed for multiscale elliptic equations \cite{LiXuZhangMscaleDNN}, neural homogenization-based physics-informed neural networks have been proposed to improve PINN accuracy for multiscale problems \cite{LeungLinZhangNHPINN}, and learning-based approaches have been developed for homogenized constitutive laws and cell-problem solution operators \cite{BhattacharyaKovachkiRajanStuartTrautner}. Physics-informed neural networks have also been applied to first-order two-scale asymptotic homogenization \cite{SoyarslanPradas}. These works demonstrate the potential of machine learning in multiscale modeling. However, most existing learning-based homogenization methods either focus on leading-order or first-order homogenization, or rely on stochastic sampling of high-dimensional loss functions. For high-order locally periodic expansions, the successive cell problems and macroscopic correction equations involve many high-dimensional integrals whose accurate evaluation is crucial.

Tensor neural networks provide a natural way to address this difficulty. The tensor product structure enables high-dimensional integrals of TNN functions to be reduced to combinations of one-dimensional quadratures, leading to deterministic and highly accurate integration without Monte Carlo sampling \cite{WangJinXieTNN, WangLinLiaoLiuXie, LinLiuXieTNNMultiscale}. This feature is especially attractive for high-order homogenization, because each corrector depends on integrals of previously computed correctors, and numerical quadrature errors can propagate through the hierarchy. The present paper combines the high-order locally periodic asymptotic expansion with a TNN-based solver, using the tensor structure to compute the loss functions and homogenized quantities accurately. The analytical results allow general smooth uniformly elliptic matrix coefficients, whereas the tensorized numerical contractions require each coefficient entry and assembled right-hand side to possess a finite or accurately truncated tensor-product representation. Arbitrary data may first be tensorized, but the additional approximation error of that preprocessing step is not analyzed here.

The main contributions of this work are as follows.

\begin{itemize}
    \item We derive a recursive high-order asymptotic expansion for locally periodic elliptic equations with coefficients $A(x,x/\varepsilon)$. The hierarchy consists of parameterized cell problems, macroscopic corrector equations, and oscillatory correction terms. The formulation is designed to be directly computable.

    \item We prove a high-order $H^1$ convergence estimate for the partial expansion in boundary-layer-free settings. The proof relies on the compatibility of the recursive cell problems, a precise residual identity, and a zero-mean oscillation estimate that yields an $\varepsilon$-order gain in the $H^{-1}$ norm.

    \item We design a tensor-neural-network method for solving the high-dimensional corrector hierarchy with tensor-structured coefficient entries and data. The method uses deterministic tensorized quadrature to evaluate the high-dimensional integrals appearing in the variational losses, cell averages, homogenized coefficients, and high-order source terms.

    \item We provide numerical experiments for locally periodic coefficients to verify the accuracy of the computed correctors. The $H^1$ semi-norm tests are compared with the proved boundary-layer-free estimate, whereas the point-normalized $L^2$ results are reported as empirical convergence observations consistent with the formal expansion orders. The tests also illustrate the advantage of high-accuracy integration in preserving the asymptotic hierarchy.
\end{itemize}

The rest of this paper is organized as follows. In Section~\ref{Section_Expansion}, we introduce the locally periodic elliptic model and derive the high-order corrector hierarchy and then prove the boundary-layer-free convergence theorem. Section~\ref{Section_Method} presents the tensor-neural-network discretization and the high-accuracy quadrature strategy. Section~\ref{Section_Numerical} reports numerical experiments and convergence tests. Finally, Section~\ref{Section_Conclusion} concludes the paper and discusses future extensions to Dirichlet boundary layer corrections.

\section{High-Order Two-Scale Expansion}
\label{Section_Expansion}

\subsection{Locally periodic model and recursive correctors}
\label{subsec:problem_setting_standard_expansion}

Let $\Omega$ be either a bounded domain in $\mathbb R^n$ or the flat torus
$\mathbb T^n$, and let $Y=(0,1)^n$ be the reference periodic cell. For
$\varepsilon>0$, we consider
\begin{equation}
-\nabla\cdot\left(A\left(x,\frac{x}{\varepsilon}\right)\nabla u_\varepsilon(x)\right)
=
f(x)
\qquad \text{in }\Omega.
\label{eq:eps_problem_main}
\end{equation}
The principal boundary-layer-free setting of the convergence analysis is the
flat torus $\Omega=\mathbb T^n$. In this setting, $A$ is periodic in both $x$
and $y$, $f$ is periodic with zero mean, and $u_\varepsilon$ is the zero-mean
periodic solution. When the problem is represented on the unit cube, we use
compatible scales $\varepsilon=1/N$, $N\in\mathbb N$, so that
$A(x,x/\varepsilon)$ is well defined across the identified faces. We also
allow the exceptional Dirichlet configurations in which every truncated
expansion satisfies the boundary data exactly; the one-dimensional mechanism
for such exact matching is described below.
The coefficient matrix $A(x,y)$ is assumed to be $Y$-periodic in $y$, symmetric, and uniformly elliptic: there exist constants
$0<\lambda\le \Lambda<\infty$ such that
\begin{equation}
\lambda |\xi|^2
\le
\xi^T A(x,y)\xi
\le
\Lambda |\xi|^2,
\qquad
\forall \xi\in\mathbb R^n,\quad (x,y)\in \Omega\times Y.
\label{eq:uniform_ellipticity}
\end{equation}
This subsection derives the recursive corrector hierarchy used later by the numerical method. Only moderate regularity is needed for the formal construction; the stronger assumptions required for the convergence theorem are collected in Subsection~\ref{subsec:high_order_convergence_ideal}. Boundary conditions do not enter the algebraic derivation itself. They enter the convergence argument through the error space: the periodic setting produces no physical boundary layer, whereas a standard Dirichlet problem generally requires boundary layer correctors. Apart from the special exact-matching configurations stated explicitly below, such Dirichlet boundary layers are outside the scope of this paper.

We introduce the two-scale differential operator
\[
D_\varepsilon=\nabla_x+\frac1\varepsilon\nabla_y.
\]
For a smooth two-scale function $v(x,y)$, one has
\[
-\nabla\cdot\left(
A\left(x,\frac{x}{\varepsilon}\right)
\nabla v\left(x,\frac{x}{\varepsilon}\right)
\right)
=
\left[
\varepsilon^{-2}L_0v
+
\varepsilon^{-1}L_1v
+
L_2v
\right]_{y=x/\varepsilon},
\]
where
\begin{align}
L_0 v
&=
-\nabla_y\cdot\left(A(x,y)\nabla_y v\right),
\label{eq:L0_def}
\\
L_1 v
&=
-\nabla_y\cdot\left(A(x,y)\nabla_x v\right)
-\nabla_x\cdot\left(A(x,y)\nabla_y v\right),
\label{eq:L1_def}
\\
L_2 v
&=
-\nabla_x\cdot\left(A(x,y)\nabla_x v\right).
\label{eq:L2_def}
\end{align}

We seek an expansion of the form
\begin{equation}
u_\varepsilon(x)
\sim
u_0(x)
+
\varepsilon u_1\left(x,\frac{x}{\varepsilon}\right)
+
\varepsilon^2 u_2\left(x,\frac{x}{\varepsilon}\right)
+\cdots,
\label{eq:formal_ansatz}
\end{equation}
where each corrector is $Y$-periodic in the fast variable $y$. Substitution of \eqref{eq:formal_ansatz} into
\eqref{eq:eps_problem_main} gives the cascade
\begin{align}
L_0u_0&=0,
\label{eq:cascade_0}
\\
L_0u_1+L_1u_0&=0,
\label{eq:cascade_1}
\\
L_0u_2+L_1u_1+L_2u_0&=f,
\label{eq:cascade_2}
\\
L_0u_{j+2}+L_1u_{j+1}+L_2u_j&=0,
\qquad j\ge1.
\label{eq:cascade_j}
\end{align}
The first equation implies that $u_0$ is independent of $y$.

For $i=1,\dots,n$, let $\chi_i(x,y)$ solve the parameterized cell problem
\begin{equation}
\begin{cases}
-\nabla_y\cdot\left(A(x,y)\nabla_y\chi_i(x,y)\right)
=
\nabla_y\cdot\left(A(x,y)e_i\right),
& y\in Y,\\[2mm]
\chi_i(x,\cdot)\ \text{is }Y\text{-periodic},\\[1mm]
\displaystyle \int_Y \chi_i(x,y)\,dy=0.
\end{cases}
\label{eq:first_cell_problem}
\end{equation}
The zero-average normalization is used in the analysis. In numerical implementations, an equivalent gauge, such as fixing the value at a reference point, may also be used whenever the cell functions are continuous; the resulting constant shift is absorbed into the macroscopic component of the corrector.

Define
\[
\boldsymbol\chi=(\chi_1,\dots,\chi_n)^T.
\]
With the sign convention \eqref{eq:first_cell_problem}, the first oscillatory corrector is
\begin{equation}
\hat u_1(x,y)
=
\boldsymbol\chi(x,y)\cdot\nabla_xu_0(x).
\label{eq:hat_u1_def}
\end{equation}
The full first-order corrector is decomposed as
\begin{equation}
u_1(x,y)=u_1^*(x,y)+\tilde u_1(x),
\qquad
u_1^*=\hat u_1.
\label{eq:u1_decomposition}
\end{equation}

The effective tensor is
\begin{equation}
A^*(x)
=
\int_Y
A(x,y)\left(I+\nabla_y\boldsymbol\chi(x,y)\right)\,dy.
\label{eq:Astar_def}
\end{equation}
The solvability condition of \eqref{eq:cascade_2} yields the homogenized equation
\begin{equation}
-\nabla_x\cdot\left(A^*(x)\nabla_xu_0(x)\right)
=
f(x).
\label{eq:homogenized_equation}
\end{equation}
In the principal periodic setting, \eqref{eq:homogenized_equation} and all
subsequent macroscopic corrector equations are solved in the zero-mean
periodic space. In an ideal boundary-matching Dirichlet configuration, $u_0$
and the macroscopic correctors instead satisfy the corresponding homogeneous
Dirichlet conditions.

From the cascade equation \eqref{eq:cascade_2}, the auxiliary second-order oscillatory term $\hat u_2$ is defined by
\begin{equation}
\begin{cases}
-\nabla_y\cdot(A\nabla_y\hat u_2)
=
f
+
\nabla_x\cdot(A\nabla_xu_0)
+
\nabla_x\cdot(A\nabla_y\hat u_1)
+
\nabla_y\cdot(A\nabla_x\hat u_1),
& y\in Y,\\[2mm]
\hat u_2(x,\cdot)\ \text{is }Y\text{-periodic},\\[1mm]
\displaystyle \int_Y\hat u_2(x,y)\,dy=0.
\end{cases}
\label{eq:hat_u2_problem}
\end{equation}
Equivalently,
\[
L_0\hat u_2
=
f-L_1\hat u_1-L_2u_0.
\]

By applying the solvability condition to the cascade equation \eqref{eq:cascade_j} for $j=1$, the macroscopic first-order corrector $\tilde u_1$ is then determined by
\begin{equation}
-\nabla_x\cdot\left(A^*(x)\nabla_x\tilde u_1(x)\right)
=
\nabla_x\cdot
\left[
\int_Y
A(x,y)
\left(
\nabla_x\hat u_1(x,y)+\nabla_y\hat u_2(x,y)
\right)\,dy
\right].
\label{eq:tilde_u1_equation}
\end{equation}
The associated oscillatory correction generated by $\tilde u_1$ is
\begin{equation}
\delta u_2(x,y)
=
\boldsymbol\chi(x,y)\cdot\nabla_x\tilde u_1(x).
\label{eq:delta_u2_def}
\end{equation}
Thus
\begin{equation}
u_2^*(x,y)=\hat u_2(x,y)+\delta u_2(x,y),
\qquad
u_2(x,y)=u_2^*(x,y)+\tilde u_2(x).
\label{eq:u2_decomposition}
\end{equation}

For higher orders, suppose that $u_{k-1}$ and $u_k^*$ have already been constructed. For $k\ge2$, define
\begin{equation}
u_k^*(x,y)=\hat u_k(x,y)+\delta u_k(x,y),
\qquad
u_k(x,y)=u_k^*(x,y)+\tilde u_k(x).
\label{eq:uk_decomposition}
\end{equation}
Derived from the cascade equation \eqref{eq:cascade_j} with $j=k-1$, the auxiliary function $\hat u_{k+1}$ is defined by
\begin{equation}
\begin{cases}
-\nabla_y\cdot(A\nabla_y\hat u_{k+1})
=
\nabla_x\cdot(A\nabla_xu_{k-1})
+
\nabla_x\cdot(A\nabla_yu_k^*)
+
\nabla_y\cdot(A\nabla_xu_k^*),
& y\in Y,\\[2mm]
\hat u_{k+1}(x,\cdot)\ \text{is }Y\text{-periodic},\\[1mm]
\displaystyle \int_Y\hat u_{k+1}(x,y)\,dy=0.
\end{cases}
\label{eq:hat_ukplus1_problem}
\end{equation}
Equivalently,
\begin{equation}
L_0\hat u_{k+1}
=
-L_1u_k^*-L_2u_{k-1}.
\label{eq:hat_ukplus1_operator_form}
\end{equation}
By imposing the solvability condition on the cascade equation \eqref{eq:cascade_j} for $j=k$, the macroscopic corrector $\tilde u_k$ is defined by
\begin{equation}
-\nabla_x\cdot\left(A^*(x)\nabla_x\tilde u_k(x)\right)
=
\nabla_x\cdot
\left[
\int_Y
A(x,y)
\left(
\nabla_xu_k^*(x,y)+\nabla_y\hat u_{k+1}(x,y)
\right)\,dy
\right].
\label{eq:tilde_uk_equation}
\end{equation}
Finally,
\begin{equation}
\delta u_{k+1}(x,y)
=
\boldsymbol\chi(x,y)\cdot\nabla_x\tilde u_k(x).
\label{eq:delta_ukplus1_def}
\end{equation}

The above construction is summarized in Algorithm \ref{alg:high_order_expansion_revised}.

\begin{algorithm}[htb!]
\caption{High-order asymptotic expansion procedure}
\label{alg:high_order_expansion_revised}
\begin{algorithmic}[1]
\State Solve the first-order cell problems \eqref{eq:first_cell_problem} for $\chi_i$, $i=1,\dots,n$.
\State Compute the homogenized tensor $A^*(x)$ by \eqref{eq:Astar_def}.
\State Solve the homogenized equation \eqref{eq:homogenized_equation} for $u_0$.
\State Set $\hat u_1=\boldsymbol\chi\cdot\nabla_xu_0$ and $u_1^*=\hat u_1$.
\State Solve the auxiliary cell problem \eqref{eq:hat_u2_problem} for $\hat u_2$.
\State Solve the macroscopic corrector equation \eqref{eq:tilde_u1_equation} for $\tilde u_1$.
\State Set $u_1=u_1^*+\tilde u_1$.
\State Set $\delta u_2=\boldsymbol\chi\cdot\nabla_x\tilde u_1$ and $u_2^*=\hat u_2+\delta u_2$.
\For{$k=2,3,\dots$}
    \State Given $u_{k-1}$ and $u_k^*$, solve \eqref{eq:hat_ukplus1_problem} for $\hat u_{k+1}$.
    \State Solve \eqref{eq:tilde_uk_equation} for $\tilde u_k$.
    \State Set $\delta u_{k+1}=\boldsymbol\chi\cdot\nabla_x\tilde u_k$ and $u_{k+1}^*=\hat u_{k+1}+\delta u_{k+1}$.
    \State Set $u_k=u_k^*+\tilde u_k$.
\EndFor
\end{algorithmic}
\end{algorithm}

\subsection{Boundary-layer-free convergence analysis}
\label{subsec:high_order_convergence_ideal}

We now justify the high-order convergence mechanism generated by Algorithm~\ref{alg:high_order_expansion_revised}. The analysis uses the same corrector hierarchy and the same operators $L_0,L_1,L_2$ defined in \eqref{eq:L0_def}--\eqref{eq:L2_def}. Boundary layer corrections are not included here. Instead, we work in a boundary-layer-free setting, for instance on the flat torus, or in an ideal boundary-matching situation where the truncated expansion satisfies the same boundary condition as the exact solution.

\paragraph{Derivative notation.}
For a two-scale function $v=v(x,y)$, the symbol $\nabla_x$ always denotes differentiation with respect to the first, slow variable, and $\nabla_y$ denotes differentiation with respect to the second, fast variable. After composition with $y=x/\varepsilon$, the physical gradient with respect to the real spatial variable is denoted by $\nabla$.

For any two-scale function $v(x,y)$, define
\[
v^\varepsilon(x):=v\left(x,\frac{x}{\varepsilon}\right).
\]
Then
\[
\nabla v^\varepsilon(x)
=
\left(\nabla_x v+\frac1\varepsilon\nabla_yv\right)^\varepsilon(x).
\]
We therefore introduce the two-scale $\varepsilon$-gradient and divergence
\[
\nabla^{xy}_\varepsilon
:=
\nabla_x+\frac1\varepsilon\nabla_y,
\qquad
\operatorname{div}^{xy}_\varepsilon
:=
\nabla_x\cdot+\frac1\varepsilon\nabla_y\cdot .
\]
With this notation and the operators in \eqref{eq:L0_def}--\eqref{eq:L2_def},
\[
-\nabla\cdot\left(A^\varepsilon\nabla v^\varepsilon\right)
=
\left[
-\operatorname{div}^{xy}_\varepsilon
\left(A(x,y)\nabla^{xy}_\varepsilon v(x,y)\right)
\right]^\varepsilon.
\]
Here and below
\[
A^\varepsilon(x):=A\left(x,\frac{x}{\varepsilon}\right).
\]

\paragraph{Energy space.}
Let $V$ denote the energy space in which the error is tested. In the periodic case one may take
\[
V=
\left\{
v\in H^1_\#(\mathbb T^n):
\int_{\mathbb T^n}v(x)\,dx=0
\right\},
\]
while in a boundary-matching Dirichlet setting one may take
\[
V=H_0^1(\Omega).
\]
We assume throughout this subsection that $V$ satisfies the Poincare inequality
\begin{equation}
\|v\|_{L^2(\Omega)}
\le
C_P\|\nabla v\|_{L^2(\Omega)},
\qquad v\in V.
\label{eq:poincare_V_final}
\end{equation}
In the periodic case, the zero-average conditions imposed on $f$ and on the
right-hand sides of the cell problems are solvability conditions. They should
not be confused with the zero-average conditions imposed on the corresponding
unknowns, which only select representatives modulo additive constants.

\paragraph{Analytical normalization.}
For the theoretical analysis, all cell functions are normalized by zero $Y$-average. In particular,
\[
\int_Y \chi_i(x,y)\,dy=0,
\qquad
\int_Y \hat u_j(x,y)\,dy=0.
\]
In numerical implementation one may use an equivalent gauge condition, such as fixing the value at a reference point, provided that the induced additive constants are consistently absorbed into the macroscopic correctors $\tilde u_j$.

\begin{remark}[One-dimensional normalization and exact boundary matching]
\label{rem:1d_boundary_matching}
In one spatial dimension, consider the compatible auxiliary equation
\begin{equation}
-\frac{d}{dy}\left(a(x,y)\frac{dv}{dy}(x,y)\right)=g(x,y),
\qquad y\in(0,1),
\label{eq:1d_generic_cell}
\end{equation}
where $a$ and $g$ are periodic in $y$ and $\int_0^1g(x,y)\,dy=0$. The periodic
problem normalized by $v(x,0)=0$ is equivalent to the endpoint problem
$v(x,0)=v(x,1)=0$: periodicity gives one implication, while integration of
\eqref{eq:1d_generic_cell} gives matching endpoint fluxes for the converse.
The compatibility of the recursive cell equations allows this observation to
be applied at every order.

Let $\Omega=(0,1)$ and $\varepsilon=1/N$, $N\in\mathbb N$. Choose the
oscillatory correctors to vanish at $y=0$, absorb the associated additive
shifts into the macroscopic correctors, and impose homogeneous Dirichlet
conditions on all macroscopic components. The physical endpoints correspond
to the fast phases $0$ and $N$; hence every truncated expansion satisfies the
original boundary condition exactly and generates no boundary layer. This is
a special one-dimensional equivalence only. It neither identifies a general
macroscopic Dirichlet problem with a periodic problem nor extends to
$n\ge2$, where a scalar gauge cannot make a periodic corrector vanish on an
entire boundary face. Standard multidimensional Dirichlet problems therefore
require boundary-layer correctors.
\end{remark}

For $k\ge1$, define the partial and full truncated expansions by
\begin{align}
U_k^*(x)
&=
\sum_{j=0}^{k-1}
\varepsilon^j
u_j\left(x,\frac{x}{\varepsilon}\right)
+
\varepsilon^k
u_k^*\left(x,\frac{x}{\varepsilon}\right),
\label{eq:Uk_star_final}
\\
U_k(x)
&=
\sum_{j=0}^{k}
\varepsilon^j
u_j\left(x,\frac{x}{\varepsilon}\right).
\label{eq:Uk_final}
\end{align}

In the periodic setting, zero mean in the fast variable does not by itself imply
that the composed function $u_j^*(x,x/\varepsilon)$ has zero mean in the
physical variable. We therefore introduce the physical normalization operator
\[
\mathcal N_\# w
=
w-
\frac{1}{|\Omega|}
\int_\Omega w(x)\,dx
\]
and use $\mathcal N_\# U_k^*$ and $\mathcal N_\# U_k$ as the periodic
approximations. To avoid additional notation, these normalized representatives
are still denoted by $U_k^*$ and $U_k$ below. This normalization changes neither
the residual equations nor the gradients. In the boundary-matching Dirichlet
setting, no such constant shift is made.

\begin{remark}[Gauge invariance and numerical point normalization]
\label{rem:gauge-invariance-point-normalization}
The zero-mean and point-normalized representatives of a continuous periodic
solution differ only by an additive constant. For a cell function this shift
may depend on $x$ and is absorbed consistently into the macroscopic
correctors $\tilde u_j$, leaving the complete two-scale coefficients and
truncated expansions unchanged. The proof uses the zero-mean representative
to apply the Poincare inequality \eqref{eq:poincare_V_final}; changing to a
pointwise representative leaves the residual equations, gradients, and
$H^1$ semi-norm errors unchanged.
In the numerical experiments, the FEM solution and every TNN expansion are
represented using the same pointwise normalization. The reported $L^2$
errors therefore compare the same representatives; unlike the $H^1$
semi-norm, however, the $L^2$ norm is not itself gauge invariant, and its
convergence rate does not follow from the gauge-invariant energy estimate
alone.
\end{remark}

\begin{lemma}[Compatibility of the recursive cell problems]
\label{lem:compatibility_final}
The right-hand sides of the cell problems \eqref{eq:hat_u2_problem} and \eqref{eq:hat_ukplus1_problem} have zero $Y$-average. Hence the corresponding periodic zero-mean cell problems are solvable.
\end{lemma}

\begin{proof}
For \eqref{eq:hat_u2_problem}, periodicity, the identity
$u_1^*=\boldsymbol\chi\cdot\nabla_xu_0$, and \eqref{eq:Astar_def} reduce the
$Y$-average of its right-hand side to
\[
f+\nabla_x\cdot(A^*(x)\nabla_xu_0)=0
\]
by \eqref{eq:homogenized_equation}. For the general step, periodicity gives
\[
\int_Y(L_1u_{k+1}^*+L_2u_k)\,dy
=-\nabla_x\cdot\left[
\int_Y A(\nabla_xu_k^*+\nabla_y\hat u_{k+1})\,dy
+A^*\nabla_x\tilde u_k
\right].
\]
The right-hand side vanishes by \eqref{eq:tilde_uk_equation}; hence every
subsequent cell right-hand side has zero $Y$-average.
\end{proof}

\begin{lemma}[Cascade identities]
\label{lem:cascade_final}
The correctors constructed above satisfy
\begin{align}
L_0u_0&=0,
\label{eq:cascade_identity_0_final}
\\
L_0u_1^*+L_1u_0&=0,
\label{eq:cascade_identity_1_final}
\\
L_0u_{m+1}^*+L_1u_m+L_2u_{m-1}
&=
\delta_{m1}f,
\qquad m\ge1,
\label{eq:cascade_identity_m_final}
\end{align}
where $\delta_{m1}$ is the Kronecker delta.
\end{lemma}

\begin{proof}
The first two identities follow from the independence of $u_0$ from $y$ and
the first cell problem. For $m\ge1$, the defining cell equations give
\[
L_0\hat u_{m+1}+L_1u_m^*+L_2u_{m-1}=\delta_{m1}f,
\qquad
L_0\delta u_{m+1}+L_1\tilde u_m=0.
\]
Adding them and using $u_{m+1}^*=\hat u_{m+1}+\delta u_{m+1}$ and
$u_m=u_m^*+\tilde u_m$ proves \eqref{eq:cascade_identity_m_final}.
\end{proof}

\begin{lemma}[Zero-mean oscillation estimate]
\label{lem:zero_mean_oscillation_final}
Let $g(x,y)$ be $Y$-periodic in $y$ and satisfy
\[
\int_Y g(x,y)\,dy=0
\qquad
\text{for all }x.
\]
Assume that, for some $\alpha\in(0,1)$,
\[
g\in W^{1,\infty}\left(\Omega;C^\alpha_\#(Y)\right).
\]
Then there exists a constant $C$, independent of $\varepsilon$, such that
\begin{equation}
\left|
\int_\Omega
g\left(x,\frac{x}{\varepsilon}\right)
\Phi(x)\,dx
\right|
\le
C\varepsilon
\|g\|_{W^{1,\infty}(\Omega;C^\alpha_\#(Y))}
\|\Phi\|_{H^1(\Omega)}
\label{eq:zero_mean_estimate_final}
\end{equation}
for all $\Phi\in V$. Equivalently,
\[
\left\|
g\left(x,\frac{x}{\varepsilon}\right)
\right\|_{H^{-1}(\Omega)}
\le
C\varepsilon
\|g\|_{W^{1,\infty}(\Omega;C^\alpha_\#(Y))}.
\]
\end{lemma}

\begin{proof}
For each fixed $x$, solve the periodic Poisson problem
\[
-\Delta_y\theta(x,y)=g(x,y),
\qquad
\theta(x,\cdot)\ \text{is }Y\text{-periodic},
\qquad
\int_Y\theta(x,y)\,dy=0.
\]
By periodic Schauder estimates,
\[
\|\theta\|_{W^{1,\infty}(\Omega;C^{2,\alpha}_\#(Y))}
\le
C
\|g\|_{W^{1,\infty}(\Omega;C^\alpha_\#(Y))}.
\]
Set
\[
B(x,y)=-\nabla_y\theta(x,y).
\]
Then
\[
\nabla_y\cdot B(x,y)=g(x,y),
\]
and
\[
\|B\|_{L^\infty(\Omega\times Y)}
+
\|\nabla_x\cdot B\|_{L^\infty(\Omega\times Y)}
\le
C
\|g\|_{W^{1,\infty}(\Omega;C^\alpha_\#(Y))}.
\]
Define
\[
B^\varepsilon(x)
=
B\left(x,\frac{x}{\varepsilon}\right).
\]
By the chain rule for the physical divergence,
\[
\nabla\cdot B^\varepsilon(x)
=
(\nabla_x\cdot B)\left(x,\frac{x}{\varepsilon}\right)
+
\frac1\varepsilon
(\nabla_y\cdot B)\left(x,\frac{x}{\varepsilon}\right).
\]
Hence
\[
g\left(x,\frac{x}{\varepsilon}\right)
=
\varepsilon\nabla\cdot B^\varepsilon(x)
-
\varepsilon
(\nabla_x\cdot B)\left(x,\frac{x}{\varepsilon}\right).
\]
Testing against $\Phi\in V$ and integrating by parts gives
\[
\begin{aligned}
\int_\Omega
g\left(x,\frac{x}{\varepsilon}\right)\Phi(x)\,dx
&=
-\varepsilon
\int_\Omega
B^\varepsilon(x)\cdot\nabla\Phi(x)\,dx
-
\varepsilon
\int_\Omega
(\nabla_x\cdot B)^\varepsilon(x)\Phi(x)\,dx.
\end{aligned}
\]
There is no boundary contribution in the periodic case, and the boundary term vanishes in the $H_0^1$ case. Therefore
\[
\left|
\int_\Omega
g\left(x,\frac{x}{\varepsilon}\right)\Phi(x)\,dx
\right|
\le
C\varepsilon
\|g\|_{W^{1,\infty}(\Omega;C^\alpha_\#(Y))}
\|\Phi\|_{H^1(\Omega)}.
\]
This proves the estimate.
\end{proof}

\begin{lemma}[Regularity of the recursive corrector hierarchy]
\label{lem:regularity_hierarchy_final}
Fix an integer $K\ge1$. Assume that $A$ is symmetric, uniformly elliptic, $Y$-periodic in $y$, and that for some $\alpha\in(0,1)$ and some sufficiently large integer $M_K$,
\[
A\in C^{M_K,\alpha}\left(\overline\Omega;C^{M_K,\alpha}_\#(Y)\right),
\qquad
f\in C^{M_K,\alpha}(\overline\Omega).
\]
Assume also that the macroscopic elliptic problems for $u_0$ and $\tilde u_j$, $1\le j\le K$, satisfy the corresponding Schauder estimates in the chosen energy setting.

Then the correctors generated by the above hierarchy are well-defined up to order $K+1$. Moreover, for each $1\le j\le K+1$, the functions $u_j^*$ are sufficiently regular so that
\[
u_j^*\in W^{3,\infty}\left(\Omega;C^{2,\alpha}_\#(Y)\right),
\]
and the residual functions defined in Lemma~\ref{lem:residual_identity_final} below satisfy
\[
R_K\in W^{1,\infty}\left(\Omega;C^\alpha_\#(Y)\right),
\qquad
S_K\in L^\infty(\Omega\times Y).
\]
In particular, there exists a constant $C_K$, independent of $\varepsilon$, such that
\[
\|R_K\|_{W^{1,\infty}(\Omega;C^\alpha_\#(Y))}
+
\|S_K\|_{L^\infty(\Omega\times Y)}
\le C_K.
\]
\end{lemma}

\begin{proof}
The proof is a finite induction based on parameter-dependent elliptic regularity. For each fixed $x$, the operator
\[
L_0=-\nabla_y\cdot(A(x,y)\nabla_y)
\]
is uniformly elliptic on the periodic zero-mean space. Hence, whenever $F(x,\cdot)$ has zero $Y$-average, the cell problem
\[
L_0w(x,\cdot)=F(x,\cdot),
\qquad
\int_Yw(x,y)\,dy=0,
\]
has a unique periodic solution. Periodic Schauder estimates yield
\[
\|w(x,\cdot)\|_{C^{2,\alpha}_\#(Y)}
\le
C
\|F(x,\cdot)\|_{C^\alpha_\#(Y)}.
\]
Differentiating this cell problem with respect to the slow variable $x$ gives equations for $\partial_x^\beta w$ whose right-hand sides involve derivatives of $A$, lower derivatives of $w$, and derivatives of $F$. Since $A$ is sufficiently smooth, repeated application of the same estimate gives parameter-dependent bounds in $x$.

The first-order cell functions $\chi_i$ are therefore smooth in both $x$ and $y$. It follows that $A^*(x)$ is uniformly elliptic and has the required smoothness in $x$. The regularity of $u_0$ follows from the assumed Schauder estimates for the homogenized macroscopic equation.

Assume now that all terms up to a certain order have the claimed regularity. By Lemma~\ref{lem:compatibility_final}, the right-hand side of the next cell problem has zero $Y$-average. Since this right-hand side is composed of $A$, previously constructed correctors, and their $x$- and $y$-derivatives, it has the required regularity. The parameter-dependent cell estimates then give the desired bound for the next auxiliary corrector $\hat u_{j+1}$.

The macroscopic corrector $\tilde u_j$ solves an equation of the form
\[
-\nabla_x\cdot(A^*(x)\nabla_x\tilde u_j)
=
\nabla_x\cdot G_j(x),
\]
where
\[
G_j(x)
=
\int_Y
A(x,y)
\left(
\nabla_xu_j^*(x,y)+\nabla_y\hat u_{j+1}(x,y)
\right)\,dy.
\]
The regularity of $G_j$ follows from the induction hypothesis and the estimate for $\hat u_{j+1}$. The assumed macroscopic Schauder estimate yields the corresponding regularity of $\tilde u_j$. Since
\[
\delta u_{j+1}
=
\boldsymbol\chi\cdot\nabla_x\tilde u_j,
\]
the same regularity follows for $\delta u_{j+1}$ and hence for
\[
u_{j+1}^*=\hat u_{j+1}+\delta u_{j+1}.
\]
Because only finitely many correctors are needed, choosing $M_K$ sufficiently large closes the induction up to order $K+1$.

Finally, the asserted bounds for $R_K$ and $S_K$ follow directly from their definitions in Lemma~\ref{lem:residual_identity_final}, together with the established regularity of the hierarchy.
\end{proof}

\begin{lemma}[Residual identity for the partial expansion]
\label{lem:residual_identity_final}
Let $U_k^*$ be defined by \eqref{eq:Uk_star_final}. Under the assumptions of Lemma~\ref{lem:regularity_hierarchy_final}, one has
\begin{equation}
-\nabla\cdot
\left(
A^\varepsilon\nabla U_k^*
\right)
-f
=
\varepsilon^{k-1}
R_k\left(x,\frac{x}{\varepsilon}\right)
+
\varepsilon^k
S_k\left(x,\frac{x}{\varepsilon}\right),
\label{eq:residual_identity_final}
\end{equation}
where
\begin{equation}
R_k
=
L_1u_k^*
+
L_2u_{k-1}
-
\delta_{k1}f,
\label{eq:Rk_def_final}
\end{equation}
and
\begin{equation}
S_k
=
L_2u_k^*.
\label{eq:Sk_def_final}
\end{equation}
Moreover,
\begin{equation}
\int_YR_k(x,y)\,dy=0.
\label{eq:Rk_zero_mean_final}
\end{equation}
\end{lemma}

\begin{proof}
Apply the operator decomposition \eqref{eq:L0_def}--\eqref{eq:L2_def} to
$U_k^*$ and collect equal powers of $\varepsilon$. Lemma~\ref{lem:cascade_final}
cancels all terms through order $\varepsilon^{k-2}$, leaving precisely
\[
\varepsilon^{k-1}
\bigl(L_1u_k^*+L_2u_{k-1}-\delta_{k1}f\bigr)^\varepsilon
+\varepsilon^k(L_2u_k^*)^\varepsilon,
\]
which proves \eqref{eq:residual_identity_final}. The defining equation for
$\hat u_{k+1}$ gives, for every $k\ge1$,
\[
R_k=-L_0\hat u_{k+1}.
\]
Its $Y$-average vanishes because $\hat u_{k+1}$ is periodic, proving
\eqref{eq:Rk_zero_mean_final}.
\end{proof}

\begin{theorem}[High-order $H^1$ convergence without boundary layers]
\label{thm:H1_high_order_no_boundary_final}
Fix $k\ge1$. Assume the regularity hypotheses of Lemma~\ref{lem:regularity_hierarchy_final} hold with $K=k$. Assume also that no boundary layer is generated, in the sense that
\[
u_\varepsilon-U_k^*\in V.
\]
In the periodic setting, $\varepsilon=1/N$, $N\in\mathbb N$, and $U_k^*$ is
the physically normalized representative defined above. The same compatible
scales are used in the one-dimensional exact boundary-matching setting of
Remark~\ref{rem:1d_boundary_matching}.
Then there exists a constant $C_k>0$, independent of $\varepsilon$, such that
\begin{equation}
\|u_\varepsilon-U_k^*\|_{H^1(\Omega)}
\le
C_k\varepsilon^k.
\label{eq:H1_estimate_final}
\end{equation}
\end{theorem}

\begin{proof}
Let
\[
e_k^\varepsilon=u_\varepsilon-U_k^*.
\]
By assumption, $e_k^\varepsilon\in V$. The weak formulation gives
\[
\int_\Omega
A^\varepsilon\nabla e_k^\varepsilon\cdot\nabla\Phi\,dx
=
-\left\langle
-\nabla\cdot(A^\varepsilon\nabla U_k^*)-f,
\Phi
\right\rangle
\qquad
\forall \Phi\in V.
\]
Using Lemma~\ref{lem:residual_identity_final}, we get
\[
\int_\Omega
A^\varepsilon\nabla e_k^\varepsilon\cdot\nabla\Phi\,dx
=
-\varepsilon^{k-1}
\int_\Omega R_k^\varepsilon\Phi\,dx
-
\varepsilon^k
\int_\Omega S_k^\varepsilon\Phi\,dx,
\]
where
\[
R_k^\varepsilon(x)=R_k\left(x,\frac{x}{\varepsilon}\right),
\qquad
S_k^\varepsilon(x)=S_k\left(x,\frac{x}{\varepsilon}\right).
\]

By Lemma~\ref{lem:residual_identity_final}, $R_k$ has zero $Y$-average. By Lemma~\ref{lem:regularity_hierarchy_final},
\[
R_k\in W^{1,\infty}\left(\Omega;C^\alpha_\#(Y)\right).
\]
Thus the zero-mean oscillation estimate, Lemma~\ref{lem:zero_mean_oscillation_final}, gives
\[
\left|
\int_\Omega R_k^\varepsilon\Phi\,dx
\right|
\le
C_k\varepsilon
\|\Phi\|_{H^1(\Omega)}.
\]
Moreover,
\[
\left|
\int_\Omega S_k^\varepsilon\Phi\,dx
\right|
\le
\|S_k^\varepsilon\|_{L^2(\Omega)}
\|\Phi\|_{L^2(\Omega)}
\le
|\Omega|^{1/2}
\|S_k\|_{L^\infty(\Omega\times Y)}
\|\Phi\|_{L^2(\Omega)}
\le
C_k
\|\Phi\|_{H^1(\Omega)}.
\]
Hence
\[
\left|
\int_\Omega
A^\varepsilon\nabla e_k^\varepsilon\cdot\nabla\Phi\,dx
\right|
\le
C_k\varepsilon^k
\|\Phi\|_{H^1(\Omega)}.
\]
Taking $\Phi=e_k^\varepsilon$ yields
\[
\int_\Omega
A^\varepsilon\nabla e_k^\varepsilon\cdot\nabla e_k^\varepsilon\,dx
\le
C_k\varepsilon^k
\|e_k^\varepsilon\|_{H^1(\Omega)}.
\]
By uniform ellipticity,
\[
\lambda
\|\nabla e_k^\varepsilon\|_{L^2(\Omega)}^2
\le
C_k\varepsilon^k
\|e_k^\varepsilon\|_{H^1(\Omega)}.
\]
Using the Poincare inequality \eqref{eq:poincare_V_final},
\[
\|e_k^\varepsilon\|_{H^1(\Omega)}
\le
C
\|\nabla e_k^\varepsilon\|_{L^2(\Omega)}.
\]
Therefore,
\[
\|e_k^\varepsilon\|_{H^1(\Omega)}
\le
C_k\varepsilon^k.
\]
This proves \eqref{eq:H1_estimate_final}.
\end{proof}

\begin{corollary}[$H^1$ estimate for the full expansion]
\label{cor:H1_full_expansion}
Under the assumptions of Theorem~\ref{thm:H1_high_order_no_boundary_final},
assume in addition that $\tilde u_k$ is uniformly bounded in $H^1(\Omega)$.
Then the full expansion $U_k$ defined in \eqref{eq:Uk_final} satisfies
\[
\|u_\varepsilon-U_k\|_{H^1(\Omega)}
\le
C_k\varepsilon^k.
\]
\end{corollary}

\begin{proof}
Since $U_k-U_k^*=\varepsilon^k\tilde u_k$, the triangle inequality and
Theorem~\ref{thm:H1_high_order_no_boundary_final} give
\[
\|u_\varepsilon-U_k\|_{H^1(\Omega)}
\le
\|u_\varepsilon-U_k^*\|_{H^1(\Omega)}
+
\varepsilon^k\|\tilde u_k\|_{H^1(\Omega)}
\le
C_k\varepsilon^k.
\]
In the periodic setting, $\tilde u_k$ has zero mean, so the physical
normalization introduced above preserves the identity
$U_k-U_k^*=\varepsilon^k\tilde u_k$.
\end{proof}

\begin{remark}[Conditional $L^2$ estimate for the full expansion]
\label{rem:L2_conditional_final}
The theorem above is the main convergence result used in this work. If, in addition, the same assumptions hold at order $k+1$, so that
\[
\|u_\varepsilon-U_{k+1}^*\|_{H^1(\Omega)}
\le
C_{k+1}\varepsilon^{k+1},
\]
and if
\[
\|w_{k+1}^\varepsilon\|_{L^2(\Omega)}
\le
C_{k+1},
\]
where
\[
w_{k+1}^\varepsilon
=
\begin{cases}
\mathcal N_\#\!\left[
u_{k+1}^*\left(x,x/\varepsilon\right)
\right],
& \text{in the periodic setting},\\[1mm]
u_{k+1}^*\left(x,x/\varepsilon\right),
& \text{in the boundary-matching Dirichlet setting},
\end{cases}
\]
then the full expansion satisfies
\[
\|u_\varepsilon-U_k\|_{L^2(\Omega)}
\le
C_{k+1}\varepsilon^{k+1}.
\]
Indeed,
\[
U_{k+1}^*
=
U_k
+
\varepsilon^{k+1}
w_{k+1}^\varepsilon,
\]
and hence the triangle inequality gives
\[
\begin{aligned}
\|u_\varepsilon-U_k\|_{L^2(\Omega)}
&\le
\|u_\varepsilon-U_{k+1}^*\|_{L^2(\Omega)}
+
\varepsilon^{k+1}
\|w_{k+1}^\varepsilon\|_{L^2(\Omega)}
\\
&\le
C_{k+1}\varepsilon^{k+1}.
\end{aligned}
\]
The bound on $w_{k+1}^\varepsilon$ follows, for example, from Lemma~\ref{lem:regularity_hierarchy_final} with $K=k+1$. In the numerical part, this bound can also be checked directly for the computed corrector.
\end{remark}

\begin{remark}[Scope of the estimate]
The estimate in Theorem~\ref{thm:H1_high_order_no_boundary_final} is a boundary-layer-free estimate. Its principal setting is the periodic macroscopic problem. It also applies to an ideal boundary-matching Dirichlet configuration, such as the one-dimensional construction in Remark~\ref{rem:1d_boundary_matching}. For a standard Dirichlet problem, however, $U_k^*$ generally does not satisfy the same boundary condition as $u_\varepsilon$, and the error $u_\varepsilon-U_k^*$ need not belong to $H_0^1(\Omega)$. In that case, boundary layer correctors are necessary. The present result isolates the algebraic high-order convergence mechanism of the recursive expansion and provides the theoretical foundation for the high-accuracy numerical construction of the correctors.
\end{remark}

\section{Tensor Neural Network Method}
\label{Section_Method}

In this section, we describe the TNN discretization used to compute the high-order hierarchy in Section~\ref{Section_Expansion}. The method is based on strong-form least-squares residuals, deterministic tensorized quadrature, and a sequential training chain for the correctors
\[
\chi_i\ (i=1,\ldots,d),\quad u_0,\quad \hat u_1,\quad \hat u_2,\quad
\tilde u_1,\quad \delta u_2,\quad \hat u_3,\quad \tilde u_2 .
\]
The same construction applies to higher orders, but the numerical tests in Section~\ref{Section_Numerical} stop at the third-order partial expansion.

\subsection{Brief Review of TNN Architecture}

The TNN architecture is built by tensor products of one-dimensional subnetworks. Let $z=(z_1,\ldots,z_D)$ denote either a macroscopic variable $x\in\Omega$ with $D=d$, or a two-scale variable $(x,y)\in\Omega\times Y$ with $D=2d$. A rank-$p$ TNN trial function is written as
\begin{equation}
\label{def_TNN_normed}
\Psi(z;\alpha,\theta)
=
\sum_{\ell=1}^p
\alpha_\ell
\prod_{m=1}^D
\widehat\phi_{m,\ell}(z_m;\theta_m).
\end{equation}
Here $\alpha_\ell$ are scaling coefficients and the one-dimensional factors $\widehat\phi_{m,\ell}$ are normalized in their one-dimensional $L^2$ norms. For two-scale functions, we use separate one-dimensional subnetworks for the slow variables and the fast variables. All subnetworks in the reported experiments use the sine activation function, and derivatives are evaluated by automatic differentiation.

Periodicity is imposed as a hard architectural constraint, following \cite{LinLiuXieTNNMultiscale}. For every coordinate $z_m\in(0,1)$ that is required to be periodic, the first hidden layer uses
\[
\sin\bigl(2\pi k_{m,r}z_m+b_{m,r}\bigr),
\qquad k_{m,r}\in\mathbb Z,
\]
where the integers $k_{m,r}$ are fixed while the biases and all subsequent network parameters are trainable. The output of the complete subnetwork is therefore one-periodic in $z_m$. For a two-scale trial function, the reference-point gauge used in the experiments is enforced by
\begin{equation}
\Psi^{\circ}(x,y;\theta)
=\Psi(x,y;\theta)-\Psi(x,y^{\circ};\theta),
\qquad y^{\circ}=0,
\label{eq:tnn-point-gauge}
\end{equation}
which preserves fast-variable periodicity and gives $\Psi^{\circ}(x,y^{\circ};\theta)=0$. The zero-mean gauge used in the analysis can instead be imposed by subtracting $\int_Y\Psi(x,y;\theta)\,dy$; the two gauges differ by an additive function of $x$, which is absorbed consistently into the macroscopic corrector. For a macroscopic periodic trial function, subtracting its value at $x^{\circ}=0$ fixes the additive constant. In the one-dimensional Dirichlet experiment, homogeneous endpoint values are imposed by the hard transformation
\[
\Phi_D(x;\theta)=x(1-x)\Phi(x;\theta).
\]
Thus neither periodicity nor the chosen normalization is imposed through a penalty term.

Theoretical approximation properties of TNNs and the associated tensorized quadrature formulas have been studied in \cite{WangJinXieTNN, WangLinLiaoLiuXie, LinLiuXieTNNMultiscale}. In the present work, the important practical point is that the tensor-product format in \eqref{def_TNN_normed} turns the high-dimensional inner products appearing in the loss functions into contractions of one-dimensional quadrature matrices.

\subsection{Strong-Form Least Squares and Scaling Coefficients}
\label{subsec:strong_form_alpha}

All corrector equations in the numerical chain are linear after the right-hand side has been assembled from previously computed terms. Let a generic equation be written as
\[
\mathcal P v = r,
\]
where $\mathcal P$ is either the cell operator in $y$ or the macroscopic homogenized operator in $x$. For a fixed TNN basis
\[
v_\theta(z)=\sum_{\ell=1}^p\alpha_\ell\varphi_\ell(z;\theta),
\]
define
\[
\Phi_\ell=\mathcal P\varphi_\ell,\qquad
M_{\ell m}=\langle \Phi_\ell,\Phi_m\rangle,\qquad
b_\ell=\langle \Phi_\ell,r\rangle,\qquad
\rho=\langle r,r\rangle .
\]
Then the least-squares residual for the scaling vector $\alpha$ is
\begin{equation}
\label{eq:alpha_least_squares}
\mathcal J(\alpha,\theta)
=
\alpha^TM\alpha-2\alpha^Tb+\rho.
\end{equation}
For fixed basis parameters $\theta$, if $M$ is positive definite, the scaling vector $\alpha$ can be eliminated analytically:
\begin{equation}
\label{eq:alpha_elimination}
\alpha=M^{-1}b,\qquad
\mathcal J_{\mathrm{red}}(\theta)
=
\rho-b^TM^{-1}b.
\end{equation}
If $M$ is only positive semidefinite, the gauge and boundary constraints must
first be imposed on the trial space; the remaining least-squares problem is
then solved with a rank-revealing factorization or the Moore--Penrose
pseudoinverse, in which case $M^{-1}$ in \eqref{eq:alpha_elimination} is
replaced by $M^\dagger$.
This elimination is useful because the nonlinear optimizer only updates the one-dimensional subnetworks. In the implementation, some routines store the vector with the opposite sign convention for the right-hand side; this gives the equivalent formula
\[
\mathcal J_{\mathrm{red}}
=
\rho+\sum_\ell(-\alpha_\ell \widetilde b_\ell).
\]
The optimizer minimizes this squared residual objective without taking a
square root, and the values reported in the network tables are final values of
$\mathcal J$ in \eqref{eq:alpha_least_squares}.
All correctors in the reported experiments use this least-squares solve for
$\alpha$; in particular, $\alpha$ is not trained by a separate strong-alpha
loss for $\hat u_3$.

\subsection{TNN Method for High-Order Asymptotic Expansion}

Let
\[
\mathcal L_y v=-\nabla_y\cdot(A\nabla_yv),
\qquad
\mathcal L_x v=-\nabla_x\cdot(A^*\nabla_xv).
\]
The tensorized numerical implementation assumes that every coefficient entry
has a finite or accurately truncated representation of the form
\begin{equation}
A_{ij}(z)=\sum_{q=1}^{r_A}\prod_{m=1}^{D}a_{ij,q,m}(z_m),
\qquad z=(x,y),
\label{eq:coefficient-tensor-format}
\end{equation}
and that the source terms assembled along the corrector hierarchy admit the
same type of representation. Under this assumption, $A$, its derivatives,
and all strong-form right-hand sides can be handled by tensor contractions and
automatic differentiation. This is a computational assumption only: the
homogenization analysis in Section~\ref{Section_Expansion} does not require
\eqref{eq:coefficient-tensor-format}. For arbitrary non-tensor data, a
preliminary tensor approximation is required, and its approximation error
would have to be added to the TNN and asymptotic-expansion errors.

For the scalar coefficient used in the numerical tests, the cell residual for $\chi_i$ is evaluated in the expanded form
\begin{equation}
\label{loss_cell_problem}
\mathcal R_{\chi_i}
=
\partial_{y_i}A
+
\nabla_y A\cdot\nabla_y\chi_i
+
A\Delta_y\chi_i,
\qquad i=1,\ldots,d,
\end{equation}
which is equivalent to the cell equation
$-\nabla_y\cdot(A(e_i+\nabla_y\chi_i))=0$.
After $\chi_i$ are obtained, the homogenized tensor is computed by \eqref{eq:Astar_def}.

The homogenized solution $u_0$ and the macroscopic correctors $\tilde u_k$ are represented by single-$x$ TNNs. The cell functions $\chi_i$, the oscillatory correctors $\hat u_j$, and the auxiliary terms $\delta u_j$ are represented by $x/y$-pair TNNs. Although $\hat u_1=\boldsymbol\chi\cdot\nabla_xu_0$ can be assembled explicitly from previously computed factors, it may also be represented by an equivalent cell equation.

Similarly, for $k\ge1$, the shift corrector may be constructed directly as
\begin{equation}
\delta u_{k+1}(x,y)
=
\boldsymbol\chi(x,y)\cdot\nabla_x\tilde u_k(x),
\label{eq:delta_uk_direct_numerical}
\end{equation}
or obtained from the equivalent cell equation
\begin{equation}
-\nabla_y\cdot\left(A(x,y)\nabla_y\delta u_{k+1}(x,y)\right)
=
\nabla_y\cdot\left(A(x,y)\nabla_x\tilde u_k(x)\right).
\label{eq:delta_uk_cell_numerical}
\end{equation}
For a scalar coefficient, the right-hand side of
\eqref{eq:delta_uk_cell_numerical} reduces to
$\nabla_yA\cdot\nabla_x\tilde u_k$. The equivalence follows directly from
the first-order cell problems for $\boldsymbol\chi$. The cell-equation
formulation provides a separate TNN representation when derivatives of
$\delta u_{k+1}$ are needed in later right-hand sides.
In the reported two-dimensional experiment, $\delta u_2$ is obtained from
the cell-equation formulation \eqref{eq:delta_uk_cell_numerical}; this gives
a lower-rank representation and reduces the cost of assembling the subsequent
corrector equations.

The quantities are computed sequentially in the order
\[
\chi_1,\ \chi_2,\ u_0,\ \hat u_1,\ \hat u_2,\ \tilde u_1,\
\delta u_2,\ \hat u_3,\ \tilde u_2.
\]
The final network architecture is selected separately for each governing
equation. The architectures and the corresponding final objective values
$\mathcal J$ used for the reported two-dimensional experiment are given in
Table~\ref{tab:2d-network-selection}.

\subsection{High-Dimensional Quadrature Scheme}
\label{Section_Integration}

Evaluating \eqref{eq:alpha_least_squares} requires many inner products over either $\Omega$ or $\Omega\times Y$. We use deterministic tensorized quadrature throughout. If two TNN functions are written in the form \eqref{def_TNN_normed}, their product integral is reduced to sums of products of one-dimensional integrals. The same principle is used for derivatives of the TNN functions and for tensor-product representations of the coefficient and its derivatives.

This deterministic integration is important in the high-order hierarchy: the right-hand side of each new corrector contains derivatives and averages of previously trained correctors, so stochastic quadrature noise would be propagated and amplified through the chain. In both the one- and two-dimensional experiments, every one-dimensional coordinate interval is divided into 200 uniform subintervals, with four Gauss points on each subinterval. This $200\times4$ rule is used to evaluate the TNN training losses and the corrector quantities in every coordinate direction.

Except for the additional initialization stage of $u_0$ described below,
every learned component is trained with 50,000 Adam iterations \cite{KingmaBaAdam}
at an initial learning rate of $0.01$, followed by 5,000 L-BFGS steps
\cite{LiuNocedalLBFGS}. For $u_0$, we first perform 20,000 Adam iterations
using the weak energy
\begin{equation}
\mathcal E_0(u_{0,\theta})
=\frac12\int_\Omega A^*(x)\nabla u_{0,\theta}\cdot
\nabla u_{0,\theta}\,dx
-\int_\Omega f(x)u_{0,\theta}(x)\,dx,
\label{eq:u0-weak-energy}
\end{equation}
then 50,000 Adam iterations using the squared strong-form residual, and
finally 5,000 L-BFGS steps using the same squared strong-form residual. During
the weak stage, the scaling vector $\alpha$ is eliminated through the
stationarity system of \eqref{eq:u0-weak-energy}; during every strong stage it
is obtained from \eqref{eq:alpha_elimination}. All experiments
use the sine activation function. The quadrature data denoted by
\texttt{quad} in the OpenPFEM output are used only in the final
post-processing of the FEM/TNN errors; they are not used for TNN training or
for the loss values reported for the correctors.

\section{Numerical Results}
\label{Section_Numerical}

In this section, we provide several examples to validate the accuracy of the proposed
TNN-based machine learning method for the locally periodic problems. All examples use
the common quadrature, activation, and optimization settings specified in
Subsection~\ref{Section_Integration}. The proposed algorithm is implemented within
the PyTorch framework \cite{PaszkePyTorch}.

To evaluate the numerical accuracy of approximate solutions, the multiscale
problem is also solved directly by the finite element method. The reference
solution $u_{\varepsilon}^{\text{FEM}}$ is computed with the parallel finite
element platform OpenPFEM. The element order, reference mesh, quadrature rule,
and sparse direct solver used for each reported benchmark are specified below.
The same fixed reference discretization is used for all values of
$\varepsilon$ within a given benchmark.

Unless a star is displayed explicitly, the $L^2$ error tables use the full
expansions $U_k$, whereas the $H^1$ semi-norm tables use the partial expansions
$U_k^*$. Their difference is precisely the highest-order macroscopic term,
\begin{equation}
U_k-U_k^*=\varepsilon^k\tilde u_k(x).
\label{eq:full-partial-numerical-difference}
\end{equation}
Because $\tilde u_k$ is independent of the fast variable, differentiating
this difference produces no factor $\varepsilon^{-1}$; its $H^1$ contribution
therefore enters at the higher order $\varepsilon^k$ under the regularity
assumptions of Section~\ref{Section_Expansion}. At third order, where
$\tilde u_3$ was not trained, the corresponding $L^2$ column is explicitly
labelled $U_3^*$ rather than $U_3$.

All $H^1$ semi-norm errors are evaluated using the physical gradient after
composition with $y=x/\varepsilon$. In particular, the chain rule gives
\begin{align*}
\nabla U_1^*
&=\nabla_xu_0+(\nabla_yu_1^*)^\varepsilon
+\varepsilon(\nabla_xu_1^*)^\varepsilon,\\
\nabla U_2^*
&=\nabla_xu_0+(\nabla_yu_1^*)^\varepsilon
+\varepsilon\bigl(\nabla_xu_1+\nabla_yu_2^*\bigr)^\varepsilon
+\varepsilon^2(\nabla_xu_2^*)^\varepsilon.
\end{align*}
These complete physical gradients, including the highest displayed
$x$-derivative terms, are used in the post-processing below.

The interpretation of the two error measures is slightly different. The
$H^1$ semi-norm is invariant under an additive change of representative and
can therefore be compared directly with the boundary-layer-free estimate in
Theorem~\ref{thm:H1_high_order_no_boundary_final}. The point-normalized
$L^2$ errors compare consistently chosen FEM and TNN representatives, but the
$L^2$ norm is not gauge invariant. Consequently, their observed slopes are
reported as empirical evidence consistent with the formal expansion orders,
not as a direct consequence of the gauge-invariant $H^1$ theorem; see
Remark~\ref{rem:gauge-invariance-point-normalization}.

\subsection{One-Dimensional Elliptic Two-Scale Dirichlet Problem}
\label{sec:1d-dirichlet}
In this subsection, we solve a one-dimensional two-scale problem (i.e., $d=1$)
with the TNN-based machine learning method.
Let $\Omega = (0,1)$ with homogeneous 
Dirichlet boundary conditions.
We consider the following problem:
\begin{equation}
\begin{cases}
-\dfrac{\partial}{\partial x}
\left(A\left(x,\dfrac{x}{\varepsilon}\right)
\dfrac{\partial u_{\varepsilon}}{\partial x}\right)=F(x),
& x\in\Omega,\\[2mm]
u_{\varepsilon}(0)=u_{\varepsilon}(1)=0.
\end{cases}
\label{ex_1D}
\end{equation}
The coefficient and source term are chosen as
\begin{equation}
A(x,y) = \sin(2\pi y) + \sin(x) + 2, \qquad F(x) = 3x.
\label{eq:1d-coefficient}
\end{equation}
Here the coefficient depends on both the slow variable $x$ and the fast variable
$y=x/\varepsilon$.

This Dirichlet experiment is the special one-dimensional boundary-matching
configuration described in Remark~\ref{rem:1d_boundary_matching}, rather than
a generic Dirichlet problem covered without a boundary layer. The auxiliary
equations in $y$ use the endpoint-normalized gauge (equivalently, the
periodic gauge with value zero at $y=0$), while the macroscopic components use
homogeneous Dirichlet conditions at $x=0$ and $x=1$. All tested scales have
the form $\varepsilon=1/N$. Hence the fast phases at the two physical
endpoints are $0$ and $N$, and every truncated expansion matches the boundary
values of $u_\varepsilon$ exactly. Notice that this statement concerns the
fast-variable auxiliary equations; the macroscopic Dirichlet problem here is
not equivalent to a periodic problem, since, for example,
$\int_0^1F(x)\,dx=3/2\ne0$.

To verify the convergence results, we use the notation
\[
U_0=u_0,
\qquad
U_1=u_0+\varepsilon u_1,
\qquad
U_2=U_1+\varepsilon^2u_2,
\qquad
U_3^*=U_2+\varepsilon^3u_3^*.
\]
Here $U_3^*$ is the third-order partial expansion defined in
\eqref{eq:Uk_star_final}; the reported training chain does not include the
macroscopic term $\tilde u_3$ required for the full expansion $U_3$.
The reference solution $u_\varepsilon^{\text{FEM}}$ is computed with
continuous piecewise-quadratic elements on $65{,}536$ uniform subintervals,
so that $h=2^{-16}$ and the unconstrained quadratic space has $131{,}073$
nodal degrees of freedom. Element assembly and error post-processing both use
the 16-point Gauss rule on every subinterval. The resulting sparse linear
systems are solved on four MPI ranks through PETSc with the SuperLU direct
solver. We report the eight compatible scales
$\varepsilon \in \{1/3, 1/5, 1/8, 1/10, 1/12, 1/15, 1/20, 1/30\}$.
At the finest reported scale, $\varepsilon=1/30$, the reference mesh contains
approximately $65{,}536/30\approx2{,}185$ elements per microscopic period.

We now examine the convergence behavior of the proposed TNN-based asymptotic 
expansion. For a fixed small scale $\varepsilon = 1/30$, Figure~\ref{u_fem_vs_tnn_eps30_comparison}(a) compares the
reference FEM solution $u_\varepsilon^{\text{FEM}}$ with the successive approximations 
$U_0$, $U_1$, $U_2$, and $U_3^*$.
As expected, the zeroth-order term $u_0$ correctly captures the global 
envelope but misses the microscopic oscillations that vary on the 
scale $\mathcal{O}(\varepsilon)$. 
The error decomposition is visualized in panels (b)--(d). Panel (b) shows that the amplitude of $u_{\varepsilon}^{\text{FEM}}-U_0$ is relatively large. Once the first-order correction term is introduced, these oscillations are reduced, and panel (d) shows a further reduction for $U_2$ and the third-order partial expansion $U_3^*$.

\begin{figure}[htbp]
  \centering
  \includegraphics[width=0.95\textwidth]{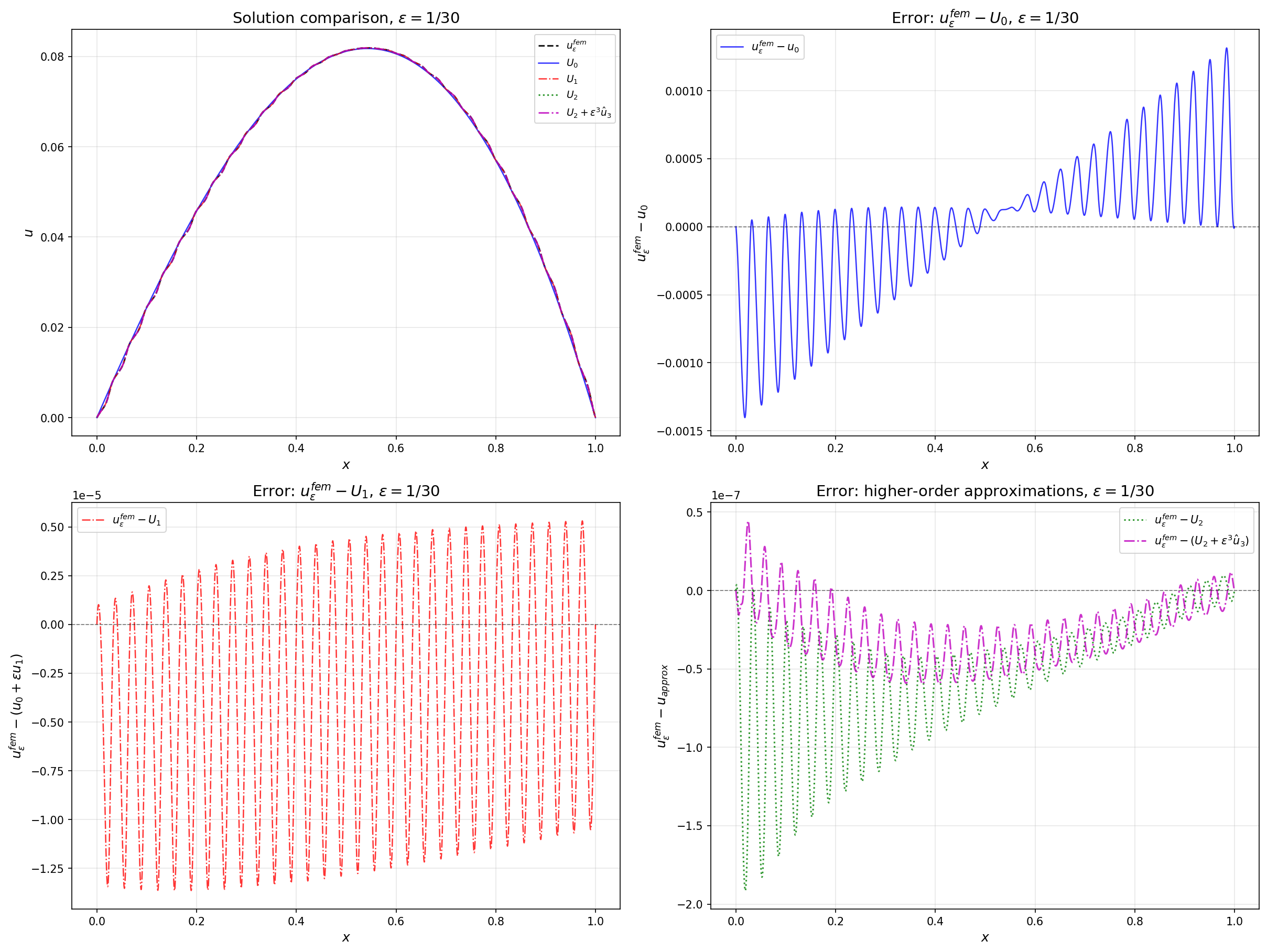}
  \caption{Comparison of the TNN approximation and finite-element reference solution at $\varepsilon=1/30$.
  (a) Solution comparison: the FEM reference solution $u_{\varepsilon}^{\text{FEM}}$, together with the asymptotic expansions
  $U_0$, $U_1$, $U_2$, and the third-order partial expansion $U_3^*$.
  (b) Error of the zeroth-order approximation: $u_{\varepsilon}^{\text{FEM}}-U_0$.
  (c) Error of the first-order approximation: $u_{\varepsilon}^{\text{FEM}}-U_1$.
  (d) Errors of the second- and third-order partial approximations:
  $u_{\varepsilon}^{\text{FEM}}-U_2$ (green) and
  $u_{\varepsilon}^{\text{FEM}}-U_3^*$ (magenta).}
  \label{u_fem_vs_tnn_eps30_comparison}
\end{figure}

Table \ref{tab:1d-L2-errors} reports the $L^2$ errors for the zeroth-order 
approximation $u_0$, the first-order approximation $u_0 + \varepsilon u_1$, 
and the second-order approximation $u_0 + \varepsilon u_1 + \varepsilon^2 u_2$.
Table \ref{tab:1d-H1-errors} shows the corresponding errors in the $H^1$ 
semi-norm. 

\begin{table}[htbp]
\caption{One-dimensional Dirichlet problem: $L^2$ errors for successive 
approximations.}
\label{tab:1d-L2-errors}
\centering
\begin{tabular}{c|ccc}
\hline
$\varepsilon$ & $\|u_\varepsilon^{\text{FEM}} - u_0\|_{L^2}$ & 
$\|u_\varepsilon^{\text{FEM}} - (u_0+\varepsilon  u_1)\|_{L^2}$ & 
$\|u_\varepsilon^{\text{FEM}} - (u_0 + \varepsilon  u_1 + \varepsilon^2  u_2) \|_{L^2}$ \\ \hline
$1/3$   & $4.874079\times10^{-3}$ & $7.277388\times10^{-4}$ & $5.497528\times10^{-5}$ \\
$1/5$   & $2.984178\times10^{-3}$ & $2.579025\times10^{-4}$ & $1.264925\times10^{-5}$ \\
$1/8$   & $1.873647\times10^{-3}$ & $9.966410\times10^{-5}$ & $3.160263\times10^{-6}$ \\
$1/10$  & $1.499790\times10^{-3}$ & $6.354273\times10^{-5}$ & $1.627626\times10^{-6}$ \\
$1/12$  & $1.250000\times10^{-3}$ & $4.4013699\times10^{-5}$ & $9.461347\times10^{-7}$ \\
$1/15$  & $9.999468\times10^{-4}$ & $2.809605\times10^{-5}$ & $4.873127\times10^{-7}$ \\
$1/20$  & $7.497766\times10^{-4}$ & $1.576447\times10^{-5}$ & $ 2.095098\times10^{-7}$ \\ 
$1/30$  & $4.996343\times10^{-4}$ & $6.990490\times10^{-6}$ & $ 6.697795\times10^{-8}$ \\ \hline
\end{tabular}
\end{table}

\begin{table}[htbp]
\caption{One-dimensional Dirichlet problem: $H^1$ semi-norm errors for 
successive approximations.}
\label{tab:1d-H1-errors}
\centering
\begin{tabular}{c|cc}
\hline
$\varepsilon$ & 
$|u_\varepsilon^{\text{FEM}} - U_1^*|_{H^1}$ & 
$|u_\varepsilon^{\text{FEM}} - U_2^*|_{H^1}$ \\ \hline
$1/3$   & $1.103554\times10^{-2}$ & $5.147922\times10^{-4}$ \\
$1/5$   & $6.731449\times10^{-3}$ & $1.899563\times10^{-4}$ \\
$1/8$   & $4.249789\times10^{-3}$ & $7.478227\times10^{-5}$ \\
$1/10$  & $3.411455\times10^{-3}$ & $4.793118\times10^{-5}$ \\
$1/12$  & $2.849368\times10^{-3}$ & $3.330484\times10^{-5}$ \\
$1/15$  & $2.284698\times10^{-3}$ & $2.132209\times10^{-5}$ \\
$1/20$  & $1.717430\times10^{-3}$ & $1.199216\times10^{-5}$ \\
$1/30$  & $1.147558\times10^{-3}$ & $5.324368\times10^{-6}$ \\ \hline
\end{tabular}
\end{table}

The one-dimensional results demonstrate the asymptotic behavior predicted by
the boundary-layer-free theory in the special matching configuration above.
The zeroth-order approximation $u_0$ achieves approximately first-order
$L^2$ convergence, consistent with the standard homogenization error estimate. 
Adding the first-order corrector $\varepsilon u_1$ improves the $L^2$ convergence rate 
to nearly second order, while the second-order approximation reaches approximately 
third-order $L^2$ convergence. Least-squares fits to the eight reported $H^1$
scales give rates approximately $0.983$ for $U_1^*$ and $1.987$ for $U_2^*$.
These results
confirm that the TNN-based corrector chain faithfully reproduces the
theoretical expansion hierarchy for this one-dimensional exact-matching
Dirichlet problem. They should not be interpreted as a boundary-layer-free
validation for general multidimensional Dirichlet problems.

\begin{figure}[htbp]
\centering
\begin{minipage}{0.48\textwidth}
\centering
\includegraphics[width=\linewidth]{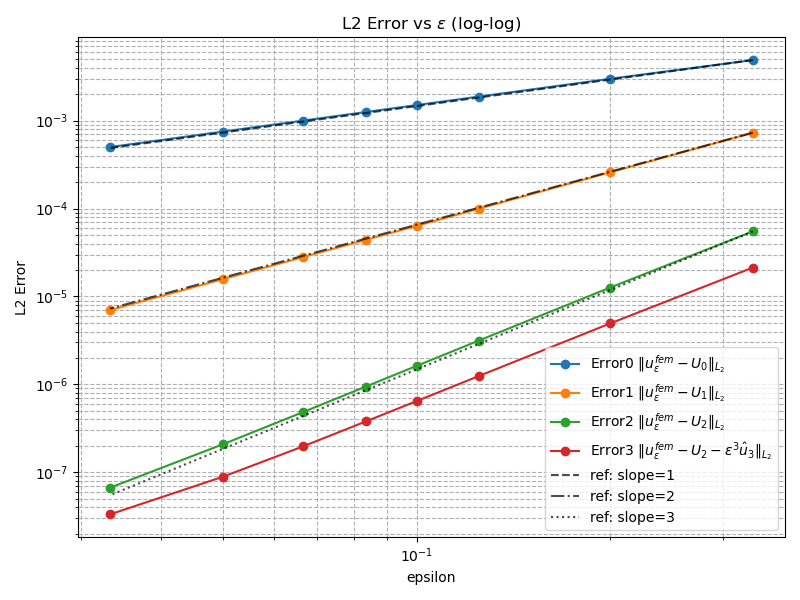}
\end{minipage}
\hfill
\begin{minipage}{0.48\textwidth}
\centering
\includegraphics[width=\linewidth]{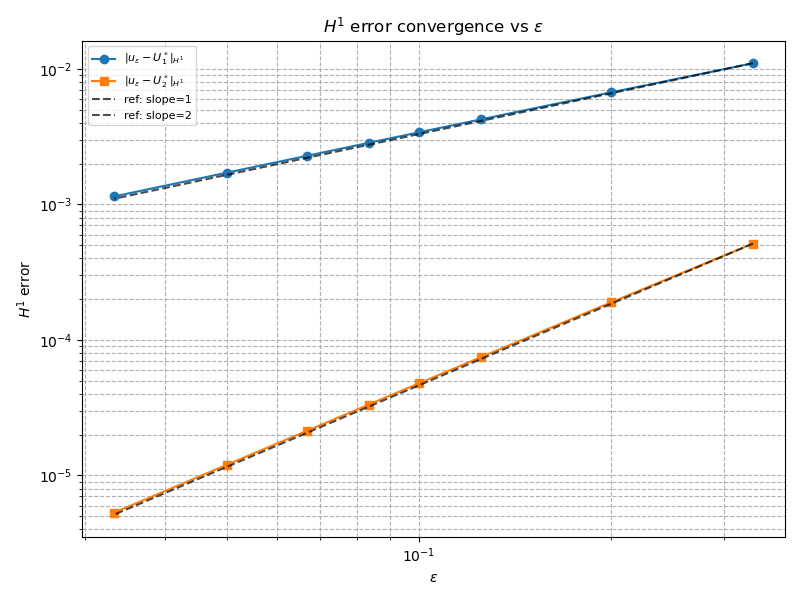}
\end{minipage}
\caption{Left: $L^2$ error of successive approximations versus $\varepsilon$ for 
    the one-dimensional Dirichlet problem. Right: $H^1$ semi-norm error of successive approximations versus 
    $\varepsilon$ for the one-dimensional Dirichlet problem.}
\label{fig:1d-errors}
\end{figure}

\subsection{Two-Dimensional Linear Multiscale Elliptic Problem with a Product-Form Coefficient}
\label{sec:2d-periodic}

In this subsection, we consider a two-dimensional problem ($d=2$) on the
unit square $\Omega=(0,1)^2$. Opposite faces are identified periodically,
and the additive constant is fixed by prescribing the value at the origin.
The model problem is
\begin{equation}
\begin{cases}
-\nabla \cdot \left( A\left(x, \frac{x}{\varepsilon}\right) \nabla u_{\varepsilon} \right) = F(x), & x \in \Omega, \\
u_{\varepsilon} \text{ is periodic across opposite faces of } \partial\Omega, \\
u_{\varepsilon}(0,0)=0.
\end{cases}
\label{ex_2D}
\end{equation}
The scalar diffusion coefficient and source term are chosen as
\begin{equation*}
\begin{aligned}
A(x,y)
&=\bigl(2+\sin(2\pi x_1)\bigr)
  \bigl(2+\sin(2\pi x_2)\bigr) \\
&\quad\times
  \bigl(1+0.5\cos(2\pi y_1)\bigr)
  \bigl(1+0.5\cos(2\pi y_2)\bigr), \\
F(x)&=30\sin(2\pi x_1)\cos(2\pi x_2).
\end{aligned}
\end{equation*}
Here $y=(y_1,y_2)=(x_1/\varepsilon,x_2/\varepsilon)$ denotes the fast
variable. Only the scalar-coefficient case is tested in this experiment; no
matrix-valued diffusion coefficient is used. Since $F$ has zero average over
$\Omega$, the periodic problem is compatible. Periodicity determines its
solution only up to an additive constant, and the condition
$u_\varepsilon(0,0)=0$ selects a unique continuous representative.

The coefficient in this benchmark factorizes as $A(x,y)=b(x)c(y)$.
Consequently, the factor $b(x)$ cancels from the first-order cell equations,
and the corresponding cell functions may be chosen independently of $x$.
Thus this experiment tests the recursive solver and tensorized quadrature for
a single product-form coefficient; a coefficient that is not factorizable as
one product $b(x)c(y)$ is tested separately in
Subsection~\ref{sec:2d-nonproduct}.

The OpenPFEM reference solution is computed with the same periodic boundary
condition and the same pointwise normalization. The additive constants in the
macroscopic terms and correctors are fixed consistently so that every reported
truncated multiscale expansion is also periodic and has value zero at the
origin. Thus the FEM solution and all TNN approximations are compared using
the same representative. The theoretical analysis is written using the
zero-mean periodic gauge; for these smooth solutions, replacing it by the
pointwise gauge changes only additive constants. As explained in
Remark~\ref{rem:gauge-invariance-point-normalization}, this leaves the residual
equations, gradients, and $H^1$ semi-norm errors exactly unchanged. The
reported $L^2$ errors are computed after applying the same pointwise
normalization to the FEM and TNN representatives. This periodic setting avoids
the boundary layers inherent in general multidimensional Dirichlet problems
and allows us to focus on the high-order convergence of the asymptotic
expansion.

An important practical consideration in the numerical implementation is that the correctors in the hierarchy are solved sequentially, and the accuracy of later correctors depends critically on the accuracy of earlier ones. Therefore, for each quantity in the expansion, we test multiple network architectures and select the one that yields the smallest loss function value. This implies that the network parameters differ for different governing equations to be solved. Table~\ref{tab:2d-network-selection} gives the final architecture and final objective value for every trained quantity in the two-dimensional experiment. A bracketed list records the number of neurons layer by layer, from the input layer to the output layer; thus the first entry is the input width, the second entry is the width of the first hidden layer, the intermediate entries are the remaining hidden-layer widths in order, and the last entry is the output width, namely the TNN rank $p$.

Each number in the third column is the final value of the loss functional
$\mathcal J$ used for the corresponding equation. For the grouped entry
$(\chi_1,\chi_2)$, the reported value is
\[
\max\{\mathcal J_{\chi_1},\mathcal J_{\chi_2}\}
=3.318646\times10^{-6}.
\]

To quantify the approximation errors of our TNN-based solver, we compute
reference solutions $u_\varepsilon^{\text{FEM}}$ with continuous
piecewise-quadratic triangular elements. Starting from the two-triangle
partition of the unit square, eight uniform refinements give $131{,}072$
triangles, characteristic mesh size $h=2^{-8}$, and $263{,}169$ quadratic
nodal degrees of freedom before imposing the periodic constraints. Opposite
boundary degrees of freedom are identified with coordinate tolerance
$10^{-10}$, and the lower-left degree of freedom is fixed to zero to remove
the constant nullspace. Assembly and error post-processing use the 36-point
triangle quadrature rule. The linear systems are solved on four MPI ranks
through PETSc with the MUMPS sparse direct solver. At the finest reported
scale, $\varepsilon=1/15$, there are approximately $256/15\approx17$ mesh
intervals per microscopic period in each coordinate direction. The numerical
experiments are performed on six scales with
\[
\varepsilon \in \{1/3,1/5,1/8,1/10,1/12,1/15\},
\]
for which both the $L^2$ norm errors and the $H^1$ semi-norm errors are reported.
The two-dimensional $L^2$ and $H^1$ error data and the corresponding plots are generated from the final FEM/TNN outputs by the post-processing scripts \texttt{output\_error\_L2.py} and \texttt{output\_error\_H1.py}, respectively. The OpenPFEM \texttt{quad} data are used only at this post-processing stage.

\begin{table}[htbp]
\caption{Two-dimensional periodic problem: final network architectures and training objective values for the corrector hierarchy.}
\label{tab:2d-network-selection}
\centering
\small
\setlength{\tabcolsep}{4pt}
\begin{tabular}{c|c|c|p{0.32\textwidth}}
\hline
Quantity & Final architecture & Final $\mathcal J$ & Notes \\
\hline
$\chi_1, \chi_2$ & $[1,2,50,50,50,20]$ & $3.318646 \times 10^{-6}$ & Maximum of the two cell-problem objectives \\
$u_0$ & $[1,2,50,50,50,20]$ &  $8.677601 \times 10^{-5}$ & Homogenized solution \\
$\hat{u}_1$ & $[1,5,80,80,80,20]$ &  $2.453737 \times 10^{-4}$ & First-order oscillatory corrector \\
$\hat{u}_2$ & $[1,8,100,100,100,20]$ &  $3.682750 \times 10^{-3}$ & Second-order oscillatory corrector \\
$\tilde{u}_1$ & $[1,8,100,100,100,20]$ & $7.499653 \times 10^{-7}$ &  First-order macroscopic corrector \\
$\delta u_2$ & $[1,8,100,100,100,20]$ &  $7.586282 \times 10^{-6}$ & Cell-equation shift corrector \\
$\hat{u}_3$ & $[1,8,150,150,150,20]$ & $1.204352 \times 10^{-2}$ & Third-order oscillatory corrector \\
$\tilde{u}_2$ & $[1,8,150,150,150,20]$ &  $6.786468 \times 10^{-3}$ & Second-order macroscopic corrector \\
\hline
\end{tabular}
\end{table}

The corrector hierarchy for the two-dimensional periodic problem follows the same structure as in the one-dimensional case. Specifically, the zeroth-order homogenized equation yields $u_0$, the first-order approximation is
\[
U_1=u_0+\varepsilon(\hat u_1+\tilde u_1),
\]
and the second-order approximation is
\[
U_2=U_1+\varepsilon^2(\hat u_2+\delta u_2+\tilde u_2).
\]
The additional complexity in the two-dimensional setting arises from the fact that the coefficient $A(x,y)$ depends on both fast variables $y_1$ and $y_2$, and the cell problems must be solved on the two-dimensional torus $Y = (0,1)^2$. The periodic boundary conditions on $\partial\Omega$ eliminate boundary layer effects, allowing us to examine the high-order convergence in the interior of the domain.

Figure \ref{fig:eps10_2x2} illustrates the error decomposition for the
two-dimensional problem at $\varepsilon=1/10$. Panel (a) shows a filled-contour
map of the second-order TNN approximation $U_2$,
which already exhibits the oscillatory fine-scale structure characteristic of the 
heterogeneous medium. Panels (b)--(d) display the spatial distribution of the 
approximation errors at different expansion orders.

\begin{figure}[htbp]
\centering
\includegraphics[width=0.98\linewidth]{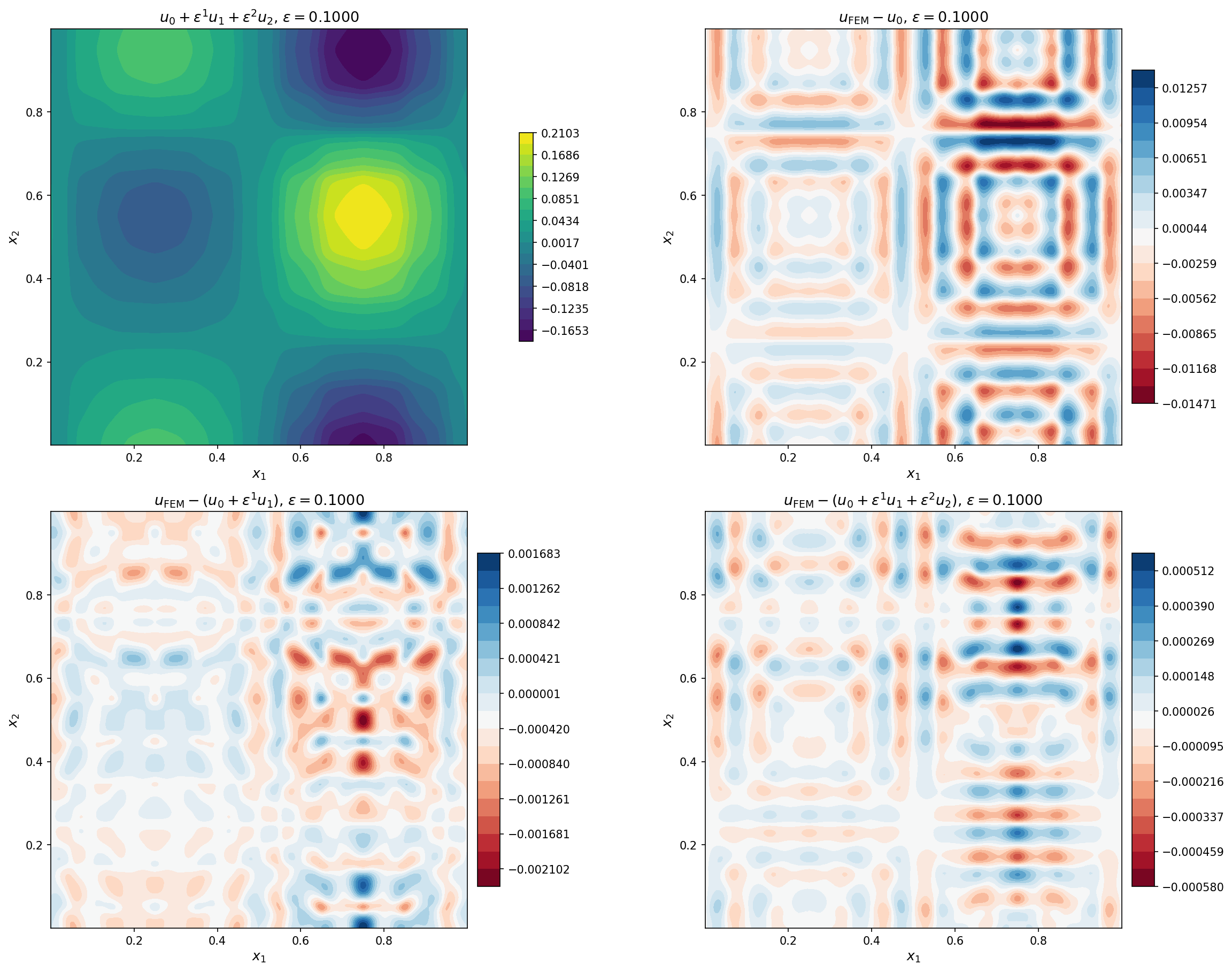}
\caption{Two-dimensional periodic problem at $\varepsilon=1/10$:
(a) filled-contour map of $U_2$;
(b) $u_{\varepsilon}^{\text{FEM}} - u_0$; 
(c) $u_{\varepsilon}^{\text{FEM}}-U_1$;
(d) $u_{\varepsilon}^{\text{FEM}}-U_2$.}
\label{fig:eps10_2x2}
\end{figure}

Tables~\ref{tab:2d-L2-errors} and~\ref{tab:2d-H1-errors} report the $L^2$ errors and $H^1$ semi-norm errors, respectively, for the successive approximations above. Additionally, we present the third-order partial $L^2$ approximation error for completeness.

\begin{table}[htbp]
\caption{Two-dimensional periodic problem: $L^2$ errors for successive approximations.}
\label{tab:2d-L2-errors}
\centering
\begin{tabular}{c|cccc}
\hline
$\varepsilon$ & $\|u_\varepsilon^{\text{FEM}} - u_0\|_{L^2}$ & 
$\|u_\varepsilon^{\text{FEM}} - U_1\|_{L^2}$ & 
$\|u_\varepsilon^{\text{FEM}} - U_2\|_{L^2}$ &
$\|u_\varepsilon^{\text{FEM}} - U_3^{*}\|_{L^2}$ \\ \hline
$1/3$   & $1.129247\times10^{-2}$ & $6.942862\times10^{-3}$ & $4.561362\times10^{-3}$ & $2.512947\times10^{-3}$ \\
$1/5$   & $7.518962\times10^{-3}$ & $2.087626\times10^{-3}$ & $9.322512\times10^{-4}$ & $3.589551\times10^{-4}$ \\
$1/8$   & $4.851752\times10^{-3}$ & $7.608214\times10^{-4}$ & $2.194178\times10^{-4}$ & $6.718247\times10^{-5}$ \\
$1/10$  & $3.910560\times10^{-3}$ & $4.761196\times10^{-4}$ & $1.108697\times10^{-4}$ & $2.988092\times10^{-5}$ \\
$1/12$  & $3.272324\times10^{-3}$ & $3.269087\times10^{-4}$ & $6.372789\times10^{-5}$ & $1.582566\times10^{-5}$ \\
$1/15$  & $2.626767\times10^{-3}$ & $2.073821\times10^{-4}$ & $3.253216\times10^{-5}$ & $7.487016\times10^{-6}$ \\ \hline
\end{tabular}
\end{table}

\begin{table}[htbp]
\caption{Two-dimensional periodic problem: $H^1$ semi-norm errors for successive approximations.}
\label{tab:2d-H1-errors}
\centering
\begin{tabular}{c|cc}
\hline
$\varepsilon$ & $|u_\varepsilon^{\text{FEM}} - U_1^*|_{H^1}$ & 
$|u_\varepsilon^{\text{FEM}} - U_2^*|_{H^1}$ \\ \hline
$1/3$   & $9.822552\times10^{-2}$ & $8.557545\times10^{-2}$ \\
$1/5$   & $5.001595\times10^{-2}$ & $3.126599\times10^{-2}$ \\
$1/8$   & $2.886039\times10^{-2}$ & $1.239209\times10^{-2}$ \\
$1/10$  & $2.261281\times10^{-2}$ & $8.009763\times10^{-3}$ \\
$1/12$  & $1.864826\times10^{-2}$ & $5.689627\times10^{-3}$ \\
$1/15$  & $1.487345\times10^{-2}$ & $4.117164\times10^{-3}$ \\ \hline
\end{tabular}
\end{table}

In the $H^1$ semi-norm, Table~\ref{tab:2d-H1-errors} shows that the second-order partial approximation $U_2^*$ further reduces the error relative to $U_1^*$. A least-squares fit on the six reported $H^1$ scales gives observed rates approximately $1.172$ for $U_1^*$ and $1.914$ for $U_2^*$.

Figure~\ref{fig:2d-errors} presents the convergence behavior of the TNN approximations in the $L^2$ norm and $H^1$ semi-norm. The $L^2$ data show approximately first-, second-, and third-order convergence for $u_0$, $U_1$, and $U_2$, respectively, while the third-order partial approximation further decreases the error. The $H^1$ data show the expected improvement from $U_1^*$ to $U_2^*$ on the reported scales.

\begin{figure}[htbp]
  \centering
  \begin{minipage}{0.48\textwidth}
    \centering
    \includegraphics[width=\linewidth]{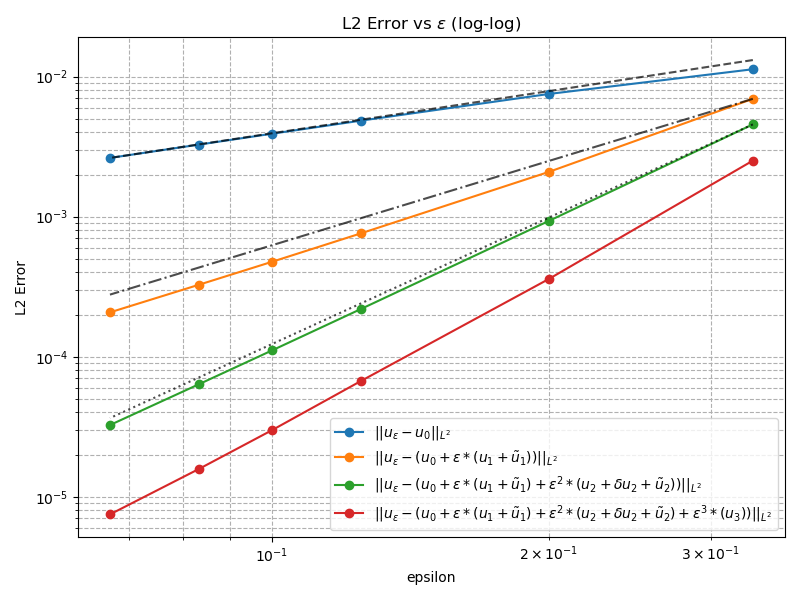}
  \end{minipage}
  \hfill 
  \begin{minipage}{0.48\textwidth}
    \centering
    \includegraphics[width=\linewidth]{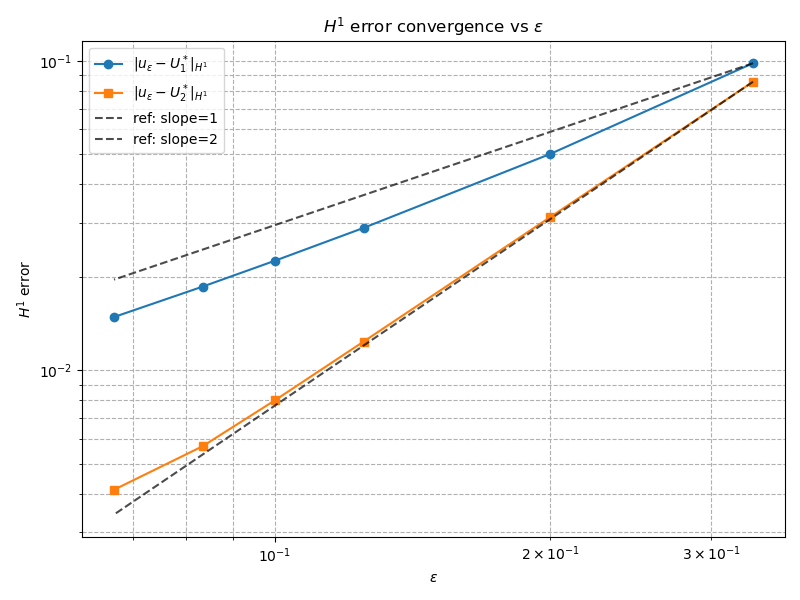}
  \end{minipage}
  \caption{Two-dimensional periodic problem: log--log scale error plots against $\varepsilon$ for successive TNN approximations. Left: $L^2$ error; Right: $H^1$ semi-norm error.}
  \label{fig:2d-errors}
\end{figure}

\subsection{Two-Dimensional Linear Multiscale Elliptic Problem with a Non-Product Coefficient}
\label{sec:2d-nonproduct}

In this subsection we consider a two-dimensional problem ($d=2$) on the unit square
$\Omega=(0,1)^2$ with periodic boundary conditions on opposite faces.
The additive constant is fixed by the value at the origin, i.e., $u_\varepsilon(0,0)=0$.
The governing equation takes the same form as~\eqref{ex_2D}, but the diffusion
coefficient is no longer factorizable as a single product $b(x)c(y)$.
Specifically,
\begin{equation}
\begin{aligned}
A(x,y) &= 3 + \sin(2\pi x_1) + \cos(2\pi x_2) + \sin(2\pi y_1)\sin(2\pi y_2),\\
F(x) &= 30\sin(2\pi x_1) + 20\cos(2\pi x_2),
\end{aligned}
\end{equation}
where $y=(y_1,y_2)=(x_1/\varepsilon,x_2/\varepsilon)$ is the fast variable.
Both $A$ and $F$ are finite sums of tensor-product terms, so they remain
compatible with the tensorized quadrature assumption
\eqref{eq:coefficient-tensor-format}. The relevant distinction from the
previous benchmark is not a lack of tensor structure: because $A$ cannot be
written as one product $b(x)c(y)$, the slow-variable factor does not cancel
from the cell equations, and the cell functions generally depend on $x$. The
factor $\sin(2\pi y_1)\sin(2\pi y_2)$ is itself a rank-one tensor product in
the two fast coordinates.

The zero-mean condition $\int_\Omega F\,\mathrm{d}x=0$ holds, ensuring
compatibility of the periodic problem. The same pointwise normalization
$u_\varepsilon(0,0)=0$ is imposed to select a unique representative. The
reference solution $u_\varepsilon^{\mathrm{FEM}}$ is computed on a sufficiently fine
mesh via OpenPFEM with the same periodic boundary conditions. All TNN approximations
and the FEM solution are compared using this identical representative, so that
only the intrinsic approximation error is measured.

As in the product-form case, the corrector hierarchy for the present problem follows
the structure
\[
U_1 = u_0 + \varepsilon(\hat u_1 + \tilde u_1),\qquad
U_2 = U_1 + \varepsilon^2(\hat u_2 + \delta u_2 + \tilde u_2).
\]
The key difference from the previous example is that the cell functions now
depend on both the slow and fast variables, even though the input data have low
tensor rank. Table~\ref{tab:2dns-network-selection} reports the final network
architectures and the corresponding final loss values for each quantity in the
expansion hierarchy. For the oscillatory correctors $\hat u_2$ and $\hat u_3$,
the loss values are noticeably higher than those in the product-form case,
which reflects the increased difficulty of resolving the slow-dependent cell
hierarchy.

\begin{table}[htbp]
\caption{Two-dimensional non-product problem: final network architectures and training objective values for the corrector hierarchy.}
\label{tab:2dns-network-selection}
\centering
\small
\setlength{\tabcolsep}{4pt}
\begin{tabular}{c|c|c|p{0.32\textwidth}}
\hline
Quantity & Final architecture & Final $\mathcal J$ & Notes \\
\hline
$\chi_1, \chi_2$ & $[1,5,80,80,20]$ & $3.51\times10^{-3}$ & Maximum of the two cell-problem objectives \\
$u_0$ & $[1,5,80,80,80,20]$ & $4.60\times10^{-2}$ & Homogenized solution \\
$\hat u_1$ & $[1,5,80,80,80,20]$ & $4.60\times10^{-2}$ & First-order oscillatory corrector \\
$\hat u_2$ & $[1,5,80,80,40]$ & $5.99\times10^{-2}$ & Second-order oscillatory corrector \\
$\tilde u_1$ & $[1,5,80,80,80,20]$ & $5.22\times10^{-6}$ & First-order macroscopic corrector \\
$\delta u_2$ & $[1,5,80,80,80,20]$ & $7.66\times10^{-5}$ & Cell-equation shift corrector \\
$\hat u_3$ & $[1,5,80,80,40]$ & $2.53\times10^{-1}$ & Third-order oscillatory corrector \\
$\tilde u_2$ & $[1,5,80,80,80,80,25]$ & $7.35\times10^{-2}$ & Second-order macroscopic corrector \\
\hline
\end{tabular}
\end{table}

The $L^2$ errors are reported on six scales,
\[
\varepsilon \in \{1/3,1/5,1/8,1/10,1/12,1/15\},
\]
whereas the $H^1$ semi-norm errors are reported on the five scales
$\{1/3,1/5,1/8,1/10,1/15\}$.

Figure~\ref{fig:2dns-eps10-2x2} displays the spatial distribution of the
approximation errors for the two-dimensional non-product problem at
$\varepsilon=1/10$. Panel (a) shows the second-order TNN approximation $U_2$,
while panels (b)--(d) show the error fields at different truncation orders.
The error maps show the additional features produced by the slow-dependent
cell functions in the corrector components.

\begin{figure}[htbp]
\centering
\includegraphics[width=0.98\linewidth]{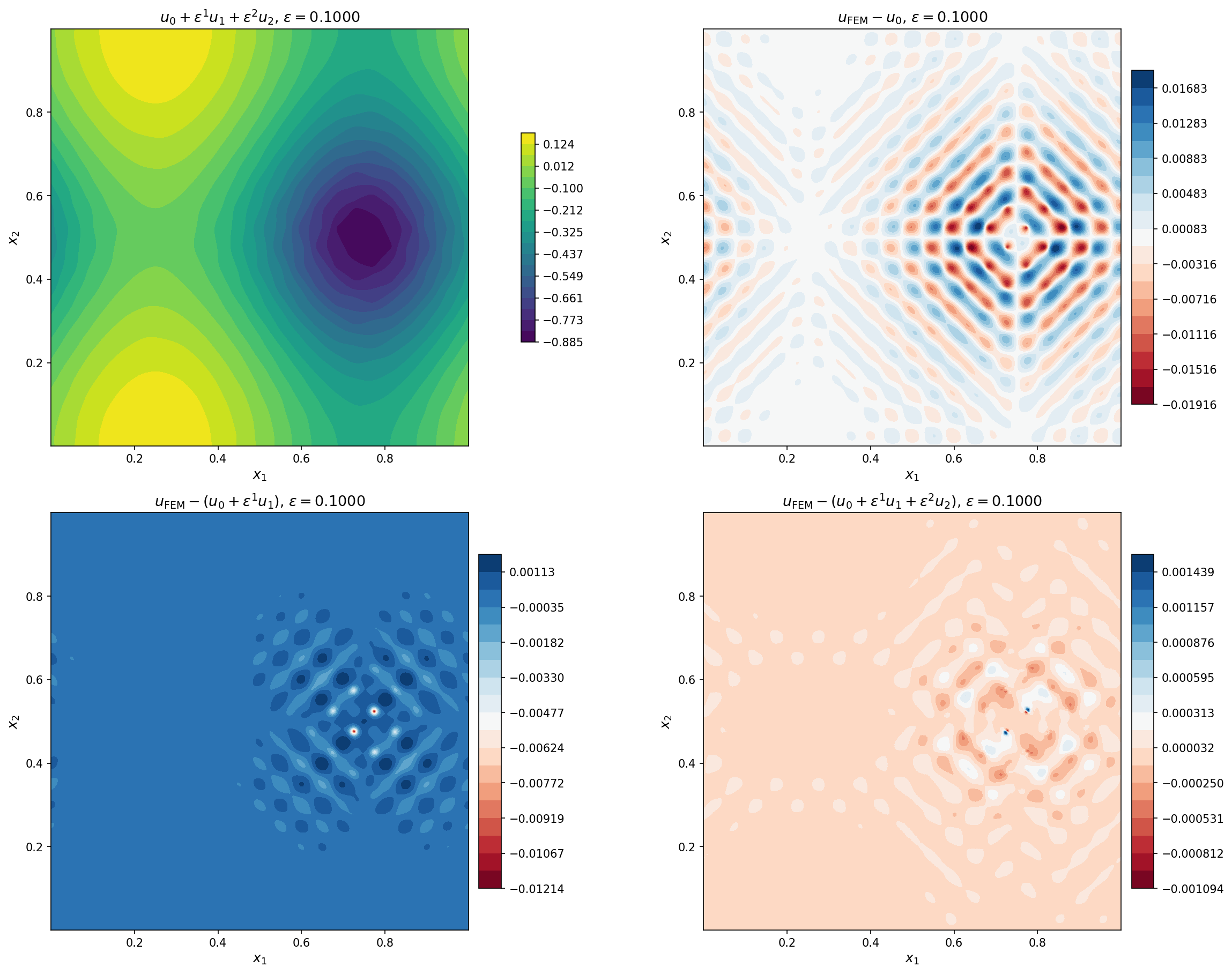}
\caption{Two-dimensional non-product problem at $\varepsilon=1/10$:
(a) second-order TNN approximation $U_2$;
(b) $u_{\varepsilon}^{\text{FEM}}-u_0$;
(c) $u_{\varepsilon}^{\text{FEM}}-U_1$;
(d) $u_{\varepsilon}^{\text{FEM}}-U_2$.}
\label{fig:2dns-eps10-2x2}
\end{figure}

Tables~\ref{tab:2dns-L2-errors} and~\ref{tab:2dns-H1-errors} collect the
$L^2$ errors and $H^1$ semi-norm errors, respectively.

\begin{table}[htbp]
\caption{Two-dimensional non-product problem: $L^2$ errors for successive approximations.}
\label{tab:2dns-L2-errors}
\centering
\begin{tabular}{c|cccc}
\hline
$\varepsilon$ & $\|u_\varepsilon^{\mathrm{FEM}} - u_0\|_{L^2}$ &
$\|u_\varepsilon^{\mathrm{FEM}} - U_1\|_{L^2}$ &
$\|u_\varepsilon^{\mathrm{FEM}} - U_2\|_{L^2}$ &
$\|u_\varepsilon^{\mathrm{FEM}} - U_3^*\|_{L^2}$ \\ \hline
$1/3$   & $1.13\times10^{-2}$ & $3.87\times10^{-3}$ & $1.99\times10^{-3}$ & $1.44\times10^{-3}$ \\
$1/5$   & $6.98\times10^{-3}$ & $1.39\times10^{-3}$ & $4.38\times10^{-4}$ & $2.93\times10^{-4}$ \\
$1/8$   & $4.40\times10^{-3}$ & $5.30\times10^{-4}$ & $1.09\times10^{-4}$ & $8.15\times10^{-5}$ \\
$1/10$  & $3.52\times10^{-3}$ & $3.39\times10^{-4}$ & $6.15\times10^{-5}$ & $4.73\times10^{-5}$ \\
$1/12$  & $2.94\times10^{-3}$ & $2.36\times10^{-4}$ & $4.07\times10^{-5}$ & $3.50\times10^{-5}$ \\
$1/15$  & $2.35\times10^{-3}$ & $1.51\times10^{-4}$ & $2.69\times10^{-5}$ & $2.49\times10^{-5}$ \\ \hline
\end{tabular}
\end{table}

\begin{table}[htbp]
\caption{Two-dimensional non-product problem: $H^1$ semi-norm errors for successive approximations.}
\label{tab:2dns-H1-errors}
\centering
\begin{tabular}{c|cc}
\hline
$\varepsilon$ & $|u_\varepsilon^{\mathrm{FEM}} - U_1^*|_{H^1}$ &
$|u_\varepsilon^{\mathrm{FEM}} - U_2^*|_{H^1}$ \\ \hline
$1/3$   & $5.71\times10^{-1}$ & $5.85\times10^{-2}$ \\
$1/5$   & $3.47\times10^{-1}$ & $2.36\times10^{-2}$ \\
$1/8$   & $2.17\times10^{-1}$ & $1.15\times10^{-2}$ \\
$1/10$  & $1.74\times10^{-1}$ & $9.96\times10^{-3}$ \\
$1/15$  & $1.17\times10^{-1}$ & $1.07\times10^{-2}$ \\ \hline
\end{tabular}
\end{table}

In the $L^2$ norm, a least-squares fit on the six reported scales yields
convergence rates of approximately $0.977$ for $u_0$, $2.02$ for $U_1$, $2.72$
for $U_2$, and $2.56$ for $U_3^*$. The first two rates agree closely with the
nominal orders $1$ and $2$; the two higher-order curves show rates below the
nominal third order on the reported scales. The third-order partial correction
still reduces the error at every listed scale.

In the $H^1$ semi-norm, the first-order partial approximation $U_1^*$ achieves a convergence rate of approximately $0.987$, consistent with the expected first-order rate. For the second-order partial approximation $U_2^*$, the least-squares fit over all five scales gives a rate of approximately $1.13$, which is below the ideal second-order rate; this discrepancy is associated with the finest scale $\varepsilon=1/15$, where the $H^1$ error of $U_2^*$ exhibits a slight increase relative to the value at $\varepsilon=1/10$ (cf. Table~\ref{tab:2dns-H1-errors}). Fitting the first four scales only yields a rate of approximately $1.51$.

Figure~\ref{fig:2d-ex3-errors} presents the convergence behavior of the TNN approximations in the $L^2$ norm and $H^1$ semi-norm for the non-product problem. The $L^2$ data show approximately first-, second-, and third-order convergence for $u_0$, $U_1$, and $U_2$, respectively, while the third-order partial approximation $U_3^*$ further reduces the error. In the $H^1$ semi-norm, $U_1^*$ converges at approximately first order, and $U_2^*$ reduces the error relative to $U_1^*$ on all reported scales, although a deviation from the ideal second-order rate is observed on the finest scale.

\begin{figure}[htbp]
  \centering
  \begin{minipage}{0.48\textwidth}
    \centering
    \includegraphics[width=\linewidth]{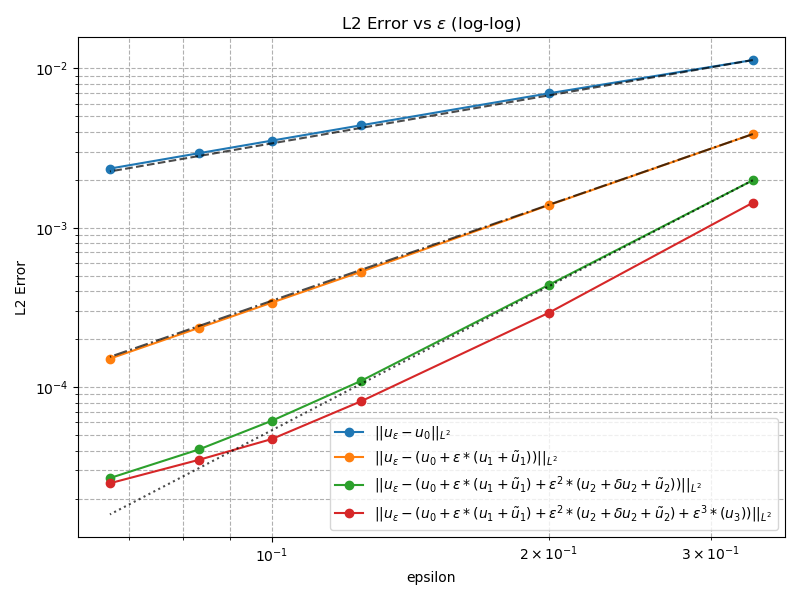}
  \end{minipage}
  \hfill 
  \begin{minipage}{0.48\textwidth}
    \centering
    \includegraphics[width=\linewidth]{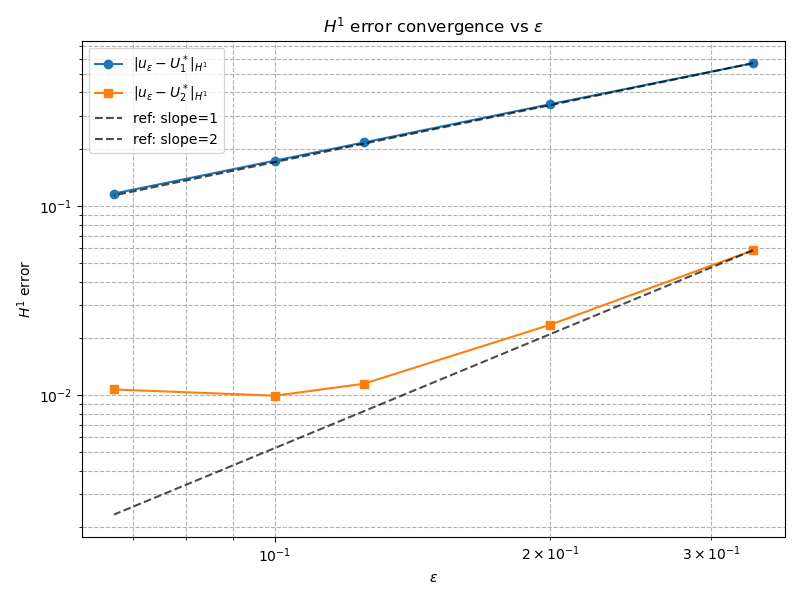}
  \end{minipage}
  \caption{Two-dimensional non-product periodic problem: log--log scale error plots against $\varepsilon$ for successive TNN approximations. Left: $L^2$ error; Right: $H^1$ semi-norm error.}
  \label{fig:2d-ex3-errors}
\end{figure}

\subsection{Three-Dimensional Linear Multiscale Elliptic Problem}
\label{sec:3d}

In this subsection we consider a three-dimensional problem ($d=3$) posed on the unit
cube $\Omega=(0,1)^3$ with periodic boundary conditions across opposite faces.
The additive constant is fixed by prescribing $u_\varepsilon(0,0,0)=0$.
The governing equation is
\begin{equation}
\begin{cases}
-\nabla \cdot \left( A\left(x, \frac{x}{\varepsilon}\right) \nabla u_\varepsilon \right) = F(x),
& x \in \Omega,\\[2mm]
u_\varepsilon \text{ is periodic on } \partial\Omega,
\end{cases}
\label{ex_3D}
\end{equation}
where the diffusion coefficient and source term are given by
\begin{equation}
\begin{aligned}
A(x,y) &= a(x,y)I_3,\\
a(x,y) &= 4+\sin(2\pi x_1)+\sin(2\pi x_2)+\sin(2\pi x_3)
+ \frac12\sin(2\pi y_1)\sin(2\pi y_2)\sin(2\pi y_3),\\[2mm]
F(x) &= 50\,\sin(2\pi x_1)\sin(2\pi x_2)\sin(2\pi x_3),
\end{aligned}
\label{ex_3d-coefficient}
\end{equation}
with $y=(y_1,y_2,y_3)=(x_1/\varepsilon,x_2/\varepsilon,x_3/\varepsilon)$.
The diffusion matrix $A$ is symmetric and uniformly positive definite, since
$a(x,y)\ge 1/2$. The oscillatory part
$0.5\sin(2\pi y_1)\sin(2\pi y_2)\sin(2\pi y_3)$ is a rank-one tensor product
in the fast coordinates. Nevertheless, the parameterized cell problem is
six-dimensional in $(x,y)$, and its solution need not inherit a rank-one
representation, making the corrector hierarchy more demanding than in lower
dimensions. The source term satisfies
$\int_\Omega F\,\mathrm{d}x=0$, ensuring compatibility of the periodic problem.
As in the two-dimensional case, the pointwise normalisation
$u_\varepsilon(0,0,0)=0$ selects a unique continuous representative.

The corrector hierarchy for the three-dimensional problem follows the same abstract
structure as in the two-dimensional case:
\[
U_1 = u_0 + \varepsilon(\hat u_1 + \tilde u_1),\qquad
U_2 = U_1 + \varepsilon^2(\hat u_2 + \delta u_2 + \tilde u_2).
\]
For brevity in the figure labels, write
$u_1=\hat u_1+\tilde u_1$ and
$u_2=\hat u_2+\delta u_2+\tilde u_2$. The oscillatory part of the first-order
coefficient is
$\hat u_1 = \chi_1\,\partial_{x_1}u_0 + \chi_2\,\partial_{x_2}u_0 + \chi_3\,\partial_{x_3}u_0$,
where $\chi_i\;(i=1,2,3)$ are the three components of the first-order cell solution.
The dependence of the cell problems on all three slow and all three fast
coordinates makes the oscillatory correctors substantially harder to approximate
in three dimensions, despite the low tensor rank of the input coefficient.

Because directly solving the fine-scale FEM problem in three dimensions at
sufficiently small $\varepsilon$ would require an enormous number of degrees of freedom
and extensive inter-process communication, we do not compute the error of the TNN
approximation against a fine-scale reference solution in this example.  Instead, we
demonstrate the qualitative behaviour of the correctors through spatial
visualisation.



Figure~\ref{fig:3d-cross-sections} presents cross-sectional visualisations at the
plane $x_1=0.3$ for three representative scales $\varepsilon\in\{1/5,1/10,1/15\}$.
Each row corresponds to one scale; the three columns display, respectively,
the total second-order approximation $U_2=u_0+\varepsilon u_1+\varepsilon^2 u_2$,
the first-order coefficient $u_1$, and the second-order coefficient $u_2$.
All fields are color-mapped on the $(x_2,x_3)$ cross-section.

\begin{figure}[htbp]
  \centering
  \includegraphics[width=0.98\linewidth]{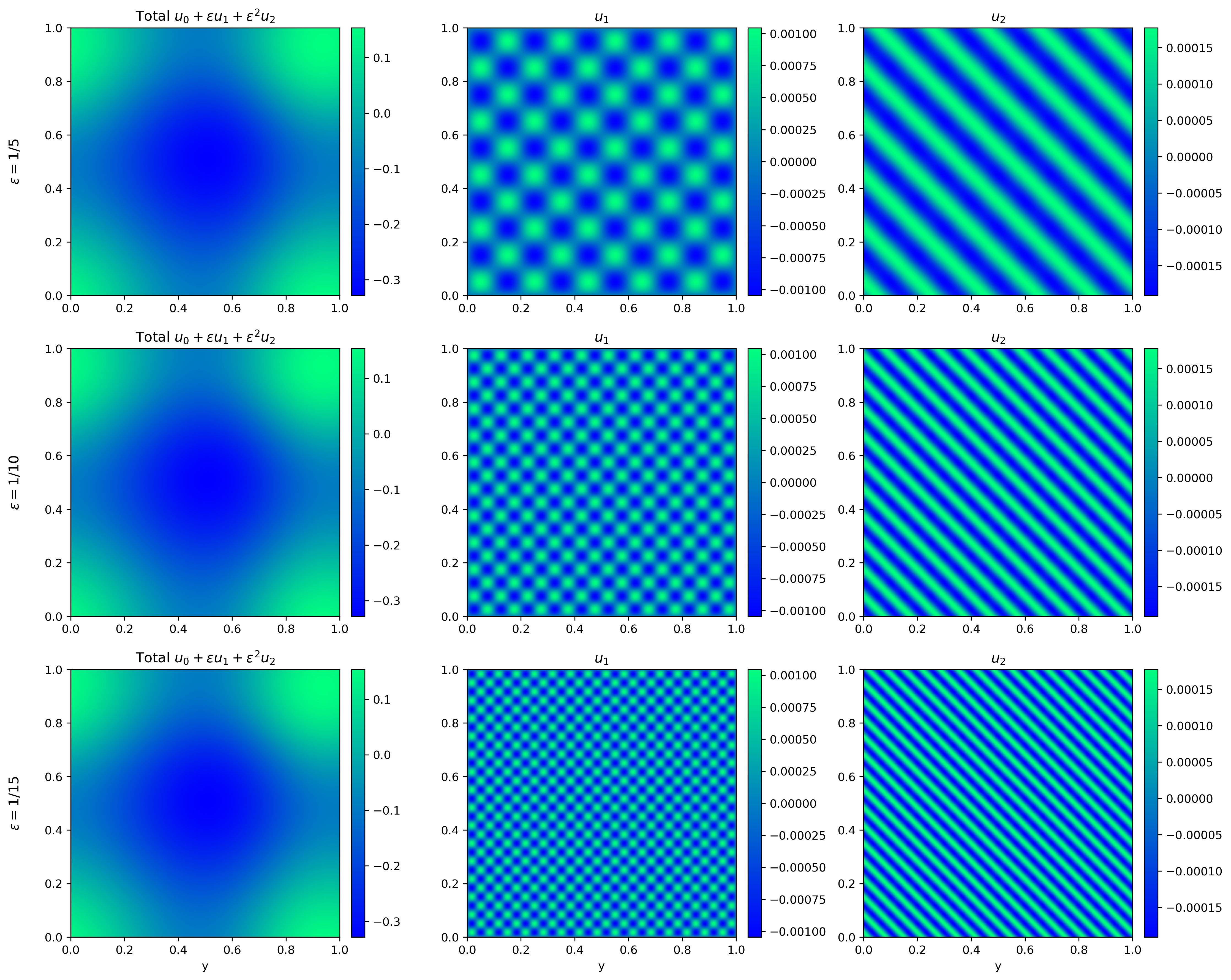}
  \caption{Three-dimensional problem: cross-sections at the plane $x_1=0.3$
for $\varepsilon=1/5$ (top row), $\varepsilon=1/10$ (middle row), and
$\varepsilon=1/15$ (bottom row). From left to right: total approximation
$U_2=u_0+\varepsilon u_1+\varepsilon^2 u_2$, first-order coefficient $u_1$,
and second-order coefficient $u_2$. The color scale is shared across
all panels within each column.}
  \label{fig:3d-cross-sections}
\end{figure}

Figure~\ref{fig:3d-cross-sections} shows the expected increase in physical
oscillation frequency as $\varepsilon$ decreases. The first- and second-order
coefficients remain bounded and visually regular on the displayed
cross-sections, and no edge-localized feature is visible in these plots.
These observations are qualitative only: without a fine-scale reference
solution they do not provide an error estimate or independently verify the
absence of boundary layers.

This experiment illustrates that the same sequential TNN construction can be
evaluated in three dimensions and that the resulting coefficients can be
visualized at several scales. A quantitative three-dimensional convergence
study remains open because it would require sufficiently resolved fine-scale
reference solutions.

\section{Conclusions}\label{Section_Conclusion}

We have developed a high-order TNN framework for locally periodic elliptic problems. The theoretical part gives a recursive two-scale hierarchy and proves the boundary-layer-free $H^1$ convergence estimate. The numerical part implements the hierarchy through strong-form TNN least-squares problems, deterministic tensorized quadrature, and a common training protocol for all learned components.

The one-dimensional boundary-matching experiment and the two-dimensional
periodic experiments show systematic reduction of the $L^2$ error as
successive correctors are added. The non-product two-dimensional coefficient
provides a first test of cell problems that retain slow-variable dependence,
although the higher-order observed rates are below their nominal asymptotic
orders on the reported scales. The three-dimensional example supplies a
qualitative visualization of the learned expansion; it is not presented as a
convergence test because no fine-scale reference solution is available.

Future work will focus on improving the accuracy and training stability of the higher-order corrector equations and on incorporating boundary layer correctors for standard Dirichlet boundary value problems.

\end{document}